\documentclass[12pt]{article}
\usepackage[utf8]{inputenc}
\usepackage{xcolor}
\usepackage{mathrsfs}
\usepackage{amsmath}
\usepackage{amssymb}
\usepackage[bookmarks=false,
 breaklinks=false,pdfborder={0 0 1},backref=section,colorlinks=true]
 {hyperref}
\hypersetup{
 citecolor=blue,urlcolor=blue,linkcolor=blue}

\makeatletter

\usepackage{graphicx}
\usepackage{amsfonts}
\usepackage{epstopdf}
\usepackage{cite}
\usepackage{todonotes}
\usepackage{graphicx}
\usepackage{subcaption}\usepackage{caption}\usepackage{xcolor}\usepackage{fullpage}
\usepackage[misc]{ifsym}
\usepackage{algorithm}
\usepackage[noend]{algpseudocode}
\usepackage{mathrsfs}
\usepackage{fontawesome5}
\providecommand{\U}[1]{\protect\rule{.1in}{.1in}}

\newtheorem{theorem}{Theorem}[section]

\newtheorem{lemma}{Lemma}[section]
\newtheorem{definition}{Definition}[section]

\newtheorem{corollary}{Corollary}[section]
\newtheorem{remark}{Remark}[section]
\newtheorem{example}{Example}[section]
\newenvironment{proof}[1][Proof]{\textbf{#1.} }{\ \rule{0.5em}{0.5em}}
\newtheorem{proposition}{Proposition}[section]

\newcommand{\soumitra}[1]{\textcolor{blue}{#1}}

\renewcommand{\soumitra}[1]{#1}

\usepackage{etoolbox}
\makeatother

\begin{document}

\title{Multi-split Multiple-sets Implicit Convex Feasibility Seeking As Split Inverse Problems}
\author{Yair Censor$^{1}$, Soumitra Dey$^{2,*}$\bigskip{}\\
$^{1}$ Department of Mathematics, University of Haifa,
Mt. Carmel,\\ Haifa 3498838, Israel\\
$^{2}$ Department of Mathematics, Hebrew University
of Jerusalem,\\ Jerusalem 9190401, Israel\\
}
\begingroup
\renewcommand{\thefootnote}{\fnsymbol{footnote}}
\footnotetext[0]{%
\begin{minipage}{\textwidth}
\small
\noindent
Yair Censor\\
\faEnvelope\ yair@math.haifa.ac.il

\medskip
\noindent
Soumitra Dey\\
\faEnvelope\ deysoumitra2012@gmail.com

\medskip
\noindent
$^{*}$ Corresponding author
\end{minipage}
}
\endgroup
\date{July 14, 2026}

\maketitle

\soumitra{This paper is dedicated to Professor Simeon Reich on the
occasion of his 80th birthday.}

\bigskip{}
\begin{abstract}
We introduce a new class of problems, termed multi-split multiple-sets
implicit convex feasibility problems (MSMS-ICFPs), define a proximity
function for it and analyze its properties under certain assumptions
on the involved general point-to-set operators. Additionally, we discuss
its properties for a special class of point-to-set operators. For these
kinds of point-to-set operators, we deploy a new iterative algorithm
to solve the associated MSMS-ICFP and show that any sequence generated
by it converges in norm to a solution of the MSMS-ICFP \soumitra{whenever
a solution exists (i.e., in the consistent case)}. \soumitra{We
also define a generalized proximity function for the MSMS-ICFP using
the Bregman distance. Based on this function, we present a generalized
simultaneous projection gradient algorithm for the split feasibility
problem with multiple output sets, by examining various related results}.
We conclude by discussing several further special cases of the MSMS-ICFP. 
\end{abstract}

\textbf{Keywords}: Implicit convex feasibility; Proximity function;
Point-to-set operators; \\Split inverse problems; Bregman distance.
\medskip{}

\textbf{2020 Mathematics Subject Classification}: 65K15; 90C25; 47N10

\section{Introduction}

\label{Sec:1}

In this paper we define the \textit{multi-split multiple-sets implicit
convex feasibility problem} (MSMS-ICFP) and investigate it. The emphasis
is on the term ``implicit'' which means that some of the constraints
sets that describe the feasibility-seeking problem are not deterministically
fixed but depend on the unknown vector itself. Put differently, these
are sets that are defined via point-to-set operators.

The MSMS-ICFP studied here is an important instance of the \textit{split
inverse problem} (SIP) that deals with two different inverse problems
formulated in two different spaces that are connected with each other
by a bounded linear transformation. We explain below these notions
and \soumitra{refer the reader to earlier published research works
for further details}.

\subsection{General Background and Motivation}

\label{subsec:General-Background-and}

The {\em convex feasibility problem} (CFP) forms the foundation
of modeling many inverse problems in various fields of mathematics
and physical sciences (see, e.g., \cite{EREM1969,YCEN1982,DEUT1992}
and references therein). It consists of finding a common point of
a finite or (countably) infinite collection of nonempty, closed, and
convex subsets of a normed linear space. We consider here the finite
case in the Euclidean space. Let $\{C_{i}\}^{r}_{i=1}$ be a finite
collection of nonempty, closed, and convex subsets of $\mathbb{R}^{n}$.
The CFP is: 
\begin{align}
\mathrm{\mathrm{Find}\:a\:point\:}x^{*}\in\mathbb{R}^{n}\mathrm{\:such\:that\:}x^{*}\in\bigcap^{r}_{i=1}C_{i}.\tag{CFP}\label{CFP:Prob}
\end{align}

Over more than the past five decades, many researchers have extensively
explored the problem \eqref{CFP:Prob} in various normed spaces and
proposed algorithms to solve it (see, for instance, \cite{YCEN1982,HHBA1996,BUTN1990,GUBI1967,HANS1988}
and references therein). Many algorithms to solve the problem \eqref{CFP:Prob}
belong to the class of projection methods (see, for example, \cite{HHBA1996}).
However, projections onto nonempty, closed, and convex subsets of
the Euclidean space or a Hilbert space can be computationally expensive
or even impossible to compute explicitly. Therefore, researchers often
focus on convex sets involved in the problem \eqref{CFP:Prob} for
which projections are easy to compute. \soumitra{Recently, in \cite{TLDI2026},
the problem \eqref{CFP:Prob} has been studied in the nonconvex and
inconsistent settings, and they have proposed a cyclic relaxed Douglas-Rachford
algorithm for it}.

In general, projection algorithms are iterative methods that use projections
onto nonempty, closed, and convex sets. The methods rely on the premise
that when a family of sets is present in a problem then projections
onto the given individual sets are easier to perform than projections
onto other sets derived from them (such as intersections, image sets
under some transformation, etc).

Projection algorithms compute projections onto nonempty, closed, and
convex sets in various ways and, therefore, the performance mechanism
of each algorithm is distinct. However, they serve to solve a variety
of problems especially in optimization theory. They are capable of
handling large-size problems with dimensions beyond those that can
be handled by other methods available presently. Moreover, simultaneous
algorithms allow parallel computations, while sequential algorithms
process one step at a time and move to the next step after having
completed all the computations at the present step. We refer the reader
to \cite{YCEN1997,cegielski-book 2012} for more details on projection
algorithms.

Behind the iterative algorithm scenario to solve nonlinear problems,
the notion of a proximity (a.k.a. gap) function plays a crucial role
in guiding the process toward a solution. This notion helps to transform
various problems in nonlinear analysis into an optimization problem
(see, for instance, \cite{cegielski-book 2012,FUKU1992,KUBO2009}
and references therein). This reformulation encourages the use of
optimization techniques and algorithms to address problems in nonlinear
analysis. Also, a proximity function may be used as a stopping criterion
for an iterative algorithm or to find error bounds of an algorithm.

\soumitra{Early fundamental references for the problem \eqref{CFP:Prob}
are \cite{LMBR1965,LMBR1967}. It was revisited by \cite{YCEN2016}}
and generalized to handle point-to-set operators, introducing the
concept of {\em implicit convex feasibility problem} (ICFP). They
demonstrated its application in image denoising within a medical imaging
situation. Motivated by the above, we further generalize here the
problem (ICFP) to the {\em multi-split multiple-sets implicit convex
feasibility problem}, which is defined below.

\subsection{Problem Definition}

\label{subsec:Problem-Definition}

Let $\mathbb{R}^{n}$ and $\mathbb{R}^{m_{j}}$ for $j=1,2,\dots,t$
be Euclidean spaces. For each $i=1,2,\dots,r$, let $C_{i}:\mathbb{R}^{n}\rightarrow2^{\mathbb{R}^{n}}$
be a point-to-set operator such that for each $x\in\mathbb{R}^{n}$,
the set $C_{i}(x)$ is nonempty, closed, and convex. For each $j=1,2,\dots,t$
and $h=1,2,\dots,\ell_{j}$, let $Q^{j}_{h}:\mathbb{R}^{m_{j}}\rightarrow2^{\mathbb{R}^{m_{j}}}$
be a point-to-set operator such that for each $y\in\mathbb{R}^{m_{j}}$,
the set $Q^{j}_{h}(y)$ is nonempty, closed, and convex. Let $A_{j}:\mathbb{R}^{n}\rightarrow\mathbb{R}^{m_{j}}$
be a bounded linear operator for each $j=1,2,\dots,t$. With these
given data the {\em multi-split multiple-sets implicit convex feasibility
problem} (MSMS-ICFP) is defined as follows: find a point $x^{*}\in\mathbb{R}^{n}$
such that 
\begin{equation}
x^{*}\in\bigcap\limits^{r}_{i=1}C_{i}(x^{*})\text{ and }A_{j}x^{*}\in\bigcap\limits^{\ell_{j}}_{h=1}Q^{j}_{h}(A_{j}x^{*}),\;\text{for each }j=1,2,\dots,t.\tag{MSMS-ICFP}\label{MSMSICFP}
\end{equation}

In particular, if we set \soumitra{$t=r$}, $\mathbb{R}^{m_{j}}=\mathbb{R}^{n}$,
$A_{j}:=I$, the identity operator, \soumitra{$Q^{j}_{h}(x)=C_{j}(x)$,
for all $x\in\mathbb{R}^{n}$, $h=1,2,\dots,\ell_{j}$, $j=1,2,\dots,t$,
in} the above formulation, then the problem \eqref{MSMSICFP} reduces
to the classical {\em implicit convex feasibility problem} (ICFP)
of \cite{YCEN2016} which consists of finding a point $x^{*}\in\mathbb{R}^{n}$
such that 
\begin{equation}
x^{*}\in\bigcap^{r}_{i=1}C_{i}(x^{*}).\tag{ICFP}\label{ICFP}
\end{equation}

\soumitra{This problem was originated by \cite{MRGH2013} in the
context of the cooperative wireless sensor network positioning problem,
and was generalized in \cite{YCEN2016}.}

Note that the above problem \eqref{MSMSICFP} can be rewritten as
follows: find a point $x^{*}\in\mathbb{R}^{n}$ such that 
\begin{equation}
x^{*}\in\bigcap\limits^{r}_{i=1}\text{Fix}(C_{i})\text{ and }A_{j}x^{*}\in\bigcap\limits^{\ell_{j}}_{h=1}\text{Fix}(Q^{j}_{h}),\;\text{for each }j=1,2,\dots,t,\tag{MSMS-ICFP}\label{MSMSICFP2}
\end{equation}
where $\text{Fix}(\cdot)$ denotes the fixed-point set of the point-to-set
operator defined below.

\subsection{The Split Inverse Problem and Related Work}

\label{subsec:The-Split-Inverse}

Let $X$ and $Y$ be two vector spaces, and $A:X\rightarrow Y$ be
a bounded linear operator. In addition, two inverse problems are involved.
The first one is $\text{IP}_{1}$, formulated in $X$ and the second
one is $\text{IP}_{2}$, formulated in $Y$. With these given data,
the \textit{split inverse problem} (SIP), introduced in \cite[Section 2]{YCEN22012}
is defined as follows: 
\[
\begin{array}{r@{\;}l}
\text{Find a point } & x^{*}\in X\text{ that solves}\text{ IP}_{1}\\
 & \text{such that}\\
\text{the point }y^{*} & =Ax^{*}\in Y\text{ solves}\text{ IP}_{2}.\tag{SIP}
\end{array}
\]

SIPs are quite general because they enable one to obtain various split
problems by making different choices of $\text{IP}_{1}$ and $\text{IP}_{2}$.
Due to their significant applications in, e.g., the intensity-modulated
radiation therapy treatment problem \cite{YCEN2005,FJAR2023} and
in economic theory \cite{JLI2019}, various types of SIPs have recently
attracted considerable attention. One notable example is the split
variational inequality problem (SVIP) \cite{YCEN22012}. For comprehensive
discussions of different types of split problems, see \cite{SDEY2024,SDEY22024,SDEY2023,GLOP2012,AMOU2011,AMOU2014,YSHE2015,SREI22022}
and references therein. The problem \eqref{MSMSICFP} studied in this
paper is an important instance of a SIP.

Current interest and activity in the field of split problems is further
evidenced by several recent relevant papers, including \cite{MRGH2013,GHTA2024,PPAT2024,ZHU2022,SSUA2022}.

\subsection{Our Contribution}

\label{subsec:Our-Contribution} \soumitra{The problem \eqref{MSMSICFP}
provides a unified structure for feasibility problems involving implicit
constraint sets in the domain space, together with multiple split
implicit constraint sets in possibly different image spaces. Our approach
combines the notions of proximity functions, projection methods,
and Bregman distances. The main contributions to this work are as
follows.}

\soumitra{
\begin{enumerate}
\item We introduce a new feasibility-seeking problem called the multi-split
multiple-sets implicit convex feasibility problem (MSMS-ICFP). In
this problem, an unknown point is required to satisfy a finite family
of implicit constraint sets in the domain space, while its images
under bounded linear operators are required to belong to prescribed
families of implicit constraint sets in the relevant image spaces.
The problem \eqref{MSMSICFP} generalizes the implicit convex feasibility
problem (ICFP). 
\item We employ a proximity function for the problem \eqref{MSMSICFP} and
establish its continuity under appropriate assumptions on the point-to-set
operators. 
\item For a special class of point-to-set operators involving translations,
scalings, and unitary transformations of fixed nonempty, closed, and
convex sets, we obtain an explicit formulation of the proximity function.
In this setting, we show that the proximity function is convex, continuously
differentiable, and its gradient is Lipschitz continuous. 
\item We propose a simultaneous projection gradient algorithm (SPGA) for
solving the problem \eqref{MSMSICFP} and prove that the sequence
generated by it converges in norm to a solution of the problem \eqref{MSMSICFP},
whenever one exists. 
\item We introduce a generalized proximity function for the problem \eqref{MSMSICFP}
by replacing Euclidean distances and orthogonal projections by Bregman
distances and Bregman projections, respectively. Since the domain
and image constraints are defined in different Euclidean spaces, we
consider a Bregman function in the domain space and other Bregman
functions in the corresponding image spaces. 
\item \soumitra{The split feasibility problem with multiple output sets is a special case of the problem \eqref{MSMSICFP} with fixed constraint sets.
We derive the differentiability and Lipschitz continuity of the gradient of the generalized proximity function for the problem under suitable assumptions. One of these is the Lipschitz continuity assumptions on the Hessian-weighted Bregman projection residual operators. We show that this assumption is more appropriate than imposing global boundedness of the projection residuals.
In particular, we give an example to show that in the Euclidean setting for arbitrary nonempty, closed, and convex constraint sets the proposed assumptions hold, even when the associated residuals are unbounded.} 
\item Based on the generalized proximity function, we establish a generalized
simultaneous projection gradient algorithm (GSPGA) for the split feasibility
problem with multiple output sets and prove that the sequence generated
by the algorithm converges in norm to a solution of the split feasibility
problem with multiple output sets, whenever one exists. 
\item We establish various special cases of the problem \eqref{MSMSICFP}
such as the split common fixed point problem, the fixed point problem,
the multiple-sets split feasibility problem, the split feasibility
problem, the convex feasibility problem, and the split common minimization
problem. 
\end{enumerate}
}


The paper is organized as follows. In Section \ref{Sec:2} we review
some existing definitions and results. In Section \ref{Sec:3} we
introduce a proximity function for the problem \eqref{MSMSICFP} and
examine its properties for general point-to-set operators. We also
discuss the properties of the proximity function for a special class
of point-to-set operators. Additionally, we present a new iterative
algorithm to solve the problem \eqref{MSMSICFP} and prove that the
sequence generated by the algorithm converges to a solution, if one
exists. We also define a generalized proximity function for the MSMS-ICFP,
and based on it, present a generalized simultaneous projection gradient
algorithm (GSPGA) for the split feasibility problem with multiple
output sets, by examining various related results. In Section \ref{other:special:MSMSICFP}
we discuss various  special cases of the problem \eqref{MSMSICFP}.
In Section \ref{Sec:conclusion} we draw a conclusion of our work.

\section{Preliminaries}

\label{Sec:2}

In this section, we recall some basic definitions and results that
will be used later in the paper. \soumitra{Throughout this paper,
$I$ denotes the identity operator on the Euclidean space whose dimension
is determined by the context. Additionally, $\|\cdot\|$ denotes the
Euclidean norm induced by the inner product $\left\langle \cdot,\cdot\right\rangle $
in the Euclidean space.}

\begin{definition}\cite[p. 28]{cegielski-book 2012}\label{Def:1}
A function $f:\mathbb{R}^{n}\rightarrow\mathbb{R}$ is said to be a
proximity function for the problem \eqref{MSMSICFP} if it satisfies
the following properties:
\begin{list}{}{%
    \setlength{\leftmargin}{3.5em}%
    \setlength{\labelwidth}{2.5em}%
    \setlength{\labelsep}{0.5em}%
    \setlength{\itemindent}{0pt}%
}
\item[$(P_{1})$] $f(x)\geq0,\;\text{for all }x\in\mathbb{R}^{n}$,\label{P1} 
\item[$(P_{2})$] $f(x^{*})=0,\;x^{*}\in\mathbb{R}^{n}$ if and only if $x^{*}$ is
a solution of the problem \eqref{MSMSICFP}.\label{P2} 
\end{list}
\end{definition}

\begin{definition}\cite{HHBA2011} \soumitra{Let $C$ be a nonempty
subset of $\mathbb{R}^{n}$. An operator $T:C\rightarrow\mathbb{R}^{n}$
is said to be Lipschitz continuous with constant $L>0$ if 
\begin{equation}
\|T(x)-T(y)\|\leq L\|x-y\|,\;\mathrm{for\;all\;}x,y\in C.
\end{equation}
} \end{definition}

\begin{lemma}\cite[Theorem 2.4]{SKAR1969}\label{lemma:strongconvex}
\soumitra{A differentiable function $f:\mathbb{R}^{n}\rightarrow\mathbb{R}$
is $\sigma$-strongly convex for some $\sigma>0$ if and only if its
gradient $\nabla f$ is $\sigma$-strongly monotone, that is, 
\begin{eqnarray}
\left\langle \nabla f(y)-\nabla f(x),y-x\right\rangle \geq\sigma\|y-x\|^{2},\;\text{\;for all}\;x,y\in\mathbb{R}^{n}.
\end{eqnarray}
} \end{lemma}

\begin{definition} Let $C$ be a nonempty subset of $\mathbb{R}^{n}$
and let $S:C\rightarrow2^{C}$ be a point-to-set operator. Then the
set of all fixed points of $S$ is denoted by $\text{Fix}(S)$ and
defined as follows: 
\begin{equation}
\text{Fix}(S):=\{x\in C\;\vert\;x\in S(x)\}.
\end{equation}
\end{definition}

\begin{definition}\cite[pp. 50--51]{HHBA2011}\label{Def:2} \soumitra{Let
$C$ be a nonempty subset of $\mathbb{R}^{n}$. The distance of a point $x\in\mathbb{R}^n$ to a set $C$ is denoted by dist$(x,C)$, i.e., 
\begin{eqnarray}
    \text{dist}(x,C)=\inf_{y\in C}\|x-y\|
\end{eqnarray}
and the set of points at which distance is attended is called the nearest point
projection of $x$ onto $C$ is defined by 
\begin{equation}
P_{C}(x):=\arg\min_{y\in C}\|x-y\|.
\end{equation}}
\end{definition}

\begin{remark}\label{rmk:proj} 
\begin{enumerate}
\item \cite[Theorem 3.16 (Projection theorem)]{HHBA2011} If $C$ is a nonempty,
closed, and convex subset of $\mathbb{R}^{n}$, then for each point
$x\in\mathbb{R}^{n}$, the projection onto $C$ exists and is unique.
Furthermore, for every $x,z\in\mathbb{R}^{n}$ 
\begin{equation}
z=P_{C}(x)\Leftrightarrow\left\langle x-z,y-z\right\rangle \leq0,\;\mathrm{for\;all\;}y\in C.\label{proj:inq1}
\end{equation}
\item Using the characterization of projection operator mentioned in \eqref{proj:inq1},
one can verify that it is nonexpansive. That is, 
\begin{equation}
\|P_{C}(x)-P_{C}(y)\|\leq\|x-y\|,\;\mathrm{for\;all\;}x,y\in\mathbb{R}^{n}.\label{proj:nonexpansive}
\end{equation}
\item From the definition of the projection operator, it is clear that $\text{Fix}(P_{C})=\{x\vert x=P_C(x)\}=\{x\vert x\in C\}=C$.
\end{enumerate}
\end{remark}

\begin{lemma}\cite[Proof of Lemma 3.2]{YNES2011}\label{lemma:translation:projection}
\soumitra{Let $C$ be a nonempty, closed, and convex subset of $\mathbb{R}^{n}$,
and let $\nu:\mathbb{R}^{n}\rightarrow\mathbb{R}^{n}$ be an operator.
Let $\varOmega:\mathbb{R}^{n}\rightarrow 2^{\mathbb{R}^{n}}$ be a point-to-set
operator such that $\varOmega(x)=\nu(x)+C=\{\nu(x)+c\vert c\in C\}$ for every $x\in\mathbb{R}^{n}$.
Then for every $x,z\in\mathbb{R}^{n}$ 
\begin{eqnarray}
P_{\varOmega(x)}(z)=\nu(x)+P_{C}(z-\nu(x)).
\end{eqnarray}
} \end{lemma}

\begin{lemma}\label{Lemma4}\cite[Lemma 2.1]{YCEN2016} Let $C$
be a nonempty, closed, and convex subset of $\mathbb{R}^{n}$. Then
for a given matrix $A$, a positive diagonal matrix $D=\alpha I$,
$\alpha>0$, and a fixed unitary matrix $U$ (that is, $UU^{T}=U^{T}U=I$),
the following is true 
\begin{eqnarray}
P_{DU(C)+Ax}(y) &  & =DUP_{C}(D^{-1}U^{T}(y-Ax))+Ax\label{ineq:PDUC}\nonumber \\
 &  & =\alpha UP_{C}\Big(\frac{1}{\alpha}U^{T}(y-Ax)\Big)+Ax,\;\mathrm{for\;all\;}x,y\in\mathbb{R}^{n},
\end{eqnarray}
where the superscript $T$ stands for the transpose and $U(C):=\{Ux\mid x\in C\}.$
\end{lemma}


\soumitra{ We deduce the following corollary from \cite[Theorem 3.4]{PLCO2005}.
Note that $\arg\min_{x\in\mathbb{R}^{n}}f(x)$ is the set of all minimizers
of $f$.}

\soumitra{\begin{corollary}\label{Co1:Grad:thm} Let $f:\mathbb{R}^{n}\rightarrow\mathbb{R}$
be a convex function with Lipschitz continuous gradient with constant
$L>0$. Assume that the solution set $G$ of the problem $\min_{x\in\mathbb{R}^{n}}f(x)$
is nonempty. Let $x^{0}\in\mathbb{R}^{n}$ be arbitrary and $\lambda\in(0,2/L)$.
Then the sequence $\{x^{k}\}_{k=0}^\infty$ generated by the gradient
method $x^{k+1}:=x^{k}-\lambda\nabla f(x^{k})$, converges in norm
to a point in $G$. \end{corollary}}

\soumitra{ \begin{proof} We know that $prox_{\lambda f}(x)=\arg\min_{y\in\mathbb{R}^{n}}\{f(y)+(1/2\lambda)\|x-y\|^{2}\}$
for $x\in\mathbb{R}^{n}$. If $f=0$ then $prox_{\lambda f}(x)=x$
for $x\in\mathbb{R}^{n}$. Under our assumptions, the proof directly
follows from \cite[Theorem 3.4 (i)]{PLCO2005} by substituting $\mathscr{H}=\mathbb{R}^{n}$,
$\lambda_{k}=1,\gamma_{k}=\lambda,a_{k}=0,b_{k}=0$ for all $k\geq0$,
$\beta=1/L$, $f_{1}=0$, $f_{2}=f$, and using the fact that $prox_{\lambda f_{1}}(x)=x$
for $x\in\mathbb{R}^{n}$. \end{proof}}

\begin{remark} \soumitra{It is worth noting that Corollary \ref{Co1:Grad:thm}
also can be deduced from \cite[Corollary 27.10]{HHBA2011} under an
appropriate choice of parameters in the Euclidean space setting.}
\end{remark}

\soumitra{It is well-known that for a given matrix $A\in\mathbb{R}^{m\times n}$,
its $2$-norm is 
\begin{eqnarray}
\|A\|_{2}=\sup\{x^{T}Ay\vert x\in\mathbb{R}^{m},y\in\mathbb{R}^{n}\text{ with }\|x\|=1=\|y\|\}.
\end{eqnarray}
Furthermore, it satisfies 
\begin{eqnarray}
\|A^{T}A\|_{2}=\|AA^{T}\|_{2}=\|A\|^{2}_{2}\;\;\text{and}\;\;\|A\|_{2}=\|A^{T}\|_{2}.
\end{eqnarray}
}


\section{Main Results: Proximity Functions and Algorithms}

\label{Sec:3}

\subsection{The Proximity Function and its Characterization}

\label{subsec:1}

Let $C_{i}:\mathbb{R}^{n}\rightarrow2^{\mathbb{R}^{n}}$ be a point-to-set
operator for $i=1,2,\dots,r$. Let, for each $x\in\mathbb{R}^{n}$,
the set $C_{i}(x)$ be nonempty, closed, and convex. For each $j=1,2,\dots,t$
and $h=1,2,\dots,\ell_{j}$, let $Q^{j}_{h}:\mathbb{R}^{m_{j}}\rightarrow2^{\mathbb{R}^{m_{j}}}$
be a finite family of point-to-set operators. Each set $Q^{j}_{h}(y)$
is nonempty, closed, and convex for every $y\in\mathbb{R}^{m_{j}}$.
Let $A_{j}:\mathbb{R}^{n}\rightarrow\mathbb{R}^{m_{j}}$ be a bounded
linear operator for each $j=1,2,\dots,t$.

Define a function $g:\mathbb{R}^{n}\rightarrow\mathbb{R}$ by 
\begin{equation}
g(x):=\frac{1}{2}\sum^{r}_{i=1}\alpha_{i}\|x-P_{C_{i}(x)}(x)\|^{2}+\frac{1}{2}\sum^{t}_{j=1}\sum^{\ell_{j}}_{h=1}\beta^{j}_{h}\|A_{j}x-P_{Q^{j}_{h}(A_{j}x)}(A_{j}x)\|^{2},\;x\in\mathbb{R}^{n},\label{Zero=00003D00003D00003D00003D000020function}
\end{equation}
where $\alpha_{i}$ and $\beta^{j}_{h}$ are positive constants for
all $h,i,j$.

From Remark \ref{rmk:proj}(1), it is evident that the function $g$
is well-defined and is a real-valued function on $\mathbb{R}^{n}$.

\begin{proposition}\label{proposition1} A point $x^{*}\in\mathbb{R}^{n}$
is a solution of the problem \eqref{MSMSICFP} if and only if it solves
the equation $g(x)=0.$ \end{proposition}

\begin{proof} Let $x^{*}\in\mathbb{R}^{n}$ be a solution of the
problem \eqref{MSMSICFP}. Then using the projection characterization
\soumitra{(Remark \ref{rmk:proj}(3))} it is straightforward to
verify that 
\begin{equation}
x^{*}=P_{C_{i}(x^{*})}(x^{*})\text{ and }A_{j}x^{*}=P_{Q^{j}_{h}(A_{j}x^{*})}(A_{j}x^{*}),\;\mathrm{for\;all\;}h,i,j.
\end{equation}
This implies that $g(x^{*})=0$.

Conversely, suppose $x^{*}\in\mathbb{R}^{n}$ solves the equation
$g(x)=0$. This implies that 
\begin{equation}
\frac{1}{2}\sum^{r}_{i=1}\alpha_{i}\|x^{*}-P_{C_{i}(x^{*})}(x^{*})\|^{2}+\frac{1}{2}\sum^{t}_{j=1}\sum^{\ell_{j}}_{h=1}\beta^{j}_{h}\|A_{j}x^{*}-P_{Q^{j}_{h}(A_{j}x^{*})}(A_{j}x^{*})\|^{2}=0\label{Zero=00003D00003D00003D00003D000020function1}
\end{equation}
yielding 
\begin{equation}
\frac{1}{2}\sum^{r}_{i=1}\alpha_{i}\|x^*-P_{C_{i}(x^{*})}(x^{*})\|^{2}=0\text{ and }\frac{1}{2}\sum^{t}_{j=1}\sum^{\ell_{j}}_{h=1}\beta^{j}_{h}\|A_{j}x^{*}-P_{Q^{j}_{h}(A_{j}x^{*})}(A_{j}x^{*})\|^{2}=0.\label{Zero=00003D00003D00003D00003D000020function2}
\end{equation}
Consequently, \soumitra{
\begin{eqnarray}
\begin{aligned} & \alpha_{i}\|x^{*}-P_{C_{i}(x^{*})}(x^{*})\|^{2}=0,\;\mathrm{for\;all\;}i=1,2,\dots,r,\;\text{and}\\
\  & \sum^{\ell_{j}}_{h=1}\beta^{j}_{h}\|A_{j}x^{*}-P_{Q^{j}_{h}(A_{j}x^{*})}(A_{j}x^{*})\|^{2}=0,\;\mathrm{for\;all\;}j=1,2,\dots,t.
\end{aligned}
\label{Zero=00003D00003D00003D00003D000020function3}
\end{eqnarray}
}

Therefore, 
\begin{equation}
x^{*}\in\bigcap^{r}_{i=1}C_{i}(x^{*})\text{ and }A_{j}x^{*}\in\bigcap^{\ell_{j}}_{h=1}Q^{j}_{h}(A_{j}x^{*}),\;\text{for each }j=1,2,\dots,t.\label{Zero=00003D00003D00003D00003D000020function4}
\end{equation}
Thus, $x^{*}\in\mathbb{R}^{n}$ is a solution of the problem \eqref{MSMSICFP}.
\end{proof}

Now, we are in the position to present the following result.

\begin{proposition}\label{proposition2} The function $g:\mathbb{R}^{n}\rightarrow\mathbb{R}$,
defined above in \eqref{Zero=00003D00003D00003D00003D000020function},
is a proximity function for the problem \eqref{MSMSICFP}. \end{proposition}

\begin{proof} By definition $g(x)$ is non-negative for all $x\in\mathbb{R}^{n}$.
From Proposition \ref{proposition1}, it is evident that $g$ satisfies
property $(P_{2})$. \soumitra{Thus, all the conditions in Definition
\ref{Def:1} are satisfied.} 
\end{proof}

\begin{corollary}\label{Cor:1} The problem \eqref{MSMSICFP} is
equivalent to the unconstrained minimization problem: 
\begin{equation}
\min_{x\in\mathbb{R}^{n}}g(x).\label{minimization}
\end{equation}
\end{corollary}


Before proceeding to prove the continuity of the proximity function
$g$ defined in \eqref{Zero=00003D00003D00003D00003D000020function},
let us establish the following Proposition.

\begin{proposition}\label{proposition3} Let $C:\mathbb{R}^{n}\rightarrow2^{\mathbb{R}^{n}}$
be a point-to-set operator such that $C(x)$ is nonempty, closed,
and convex for each $x\in\mathbb{R}^{n}$. Assume that there exists
a constant $\eta>0$ such that 
\begin{equation}
\|P_{C(x)}(z)-P_{C(y)}(z)\|\leq\eta\|x-y\|,\;\mathrm{for\;all\;}x,y,z\in\mathbb{R}^{n}.\label{inq:1}
\end{equation}
Then the function $\pi(x):=P_{C(x)}(x)$, which maps each $x\in\mathbb{R}^{n}$
to its projection onto the set $C(x)$, is continuous on $\mathbb{R}^{n}$.
\end{proposition}

\begin{proof} Let $\{x^{k}\}^{\infty}_{k=0}$ be a sequence in $\mathbb{R}^{n}$
such that $x^{k}\rightarrow x\in\mathbb{R}^{n}$ as $k\rightarrow\infty$.
Then 
\begin{eqnarray}
\|\pi(x^{k})-\pi(x)\| & = & \|P_{C(x^{k})}(x^{k})-P_{C(x)}(x)\|\nonumber \\
 & = & \|P_{C(x^{k})}(x^{k})-P_{C(x^{k})}(x)+P_{C(x^{k})}(x)-P_{C(x)}(x)\|\nonumber \\
 & \leq & \|P_{C(x^{k})}(x^{k})-P_{C(x^{k})}(x)\|+\|P_{C(x^{k})}(x)-P_{C(x)}(x)\|\nonumber \\
 & \leq & \|x^{k}-x\|+\eta\|x^{k}-x\|\nonumber \\
 & \leq & (1+\eta)\|x^{k}-x\|\rightarrow0\text{ as }k\rightarrow\infty.
\end{eqnarray}

Therefore, $\pi(x^{k})\rightarrow\pi(x)$ as $k\rightarrow\infty$
and the proof is complete. \end{proof}

\begin{remark} Let $C:\mathbb{R}^{n}\rightarrow2^{\mathbb{R}^{n}}$
be a point-to-set operator such that $C(x)$ is nonempty, closed,
and convex for each $x\in\mathbb{R}^{n}$. Assume that $C(x)=f(x)+\varOmega$,
where $\varOmega$ is a nonempty, closed, and convex subset of $\mathbb{R}^{n}$,
and $f$ is a Lipschitz continuous function on $\mathbb{R}^{n}$ into
itself with Lipschitz constant $L$. Then the inequality \eqref{inq:1}
holds with the constant $L$. For more details see, e.g., \cite[Lemma 3.2]{YNES2011}.
\end{remark}

Finally, we prove the continuity of the proximity function $g$ defined
in \eqref{Zero=00003D00003D00003D00003D000020function} under certain
assumptions on the point-to-set operators involved.

\begin{proposition}\label{proposition:continuity} Let, for each
$i=1,2,\dots,r$, $C_{i}:\mathbb{R}^{n}\rightarrow2^{\mathbb{R}^{n}}$
be a point-to-set operator such that $C_{i}(x)$ is nonempty, closed,
and convex for each $x\in\mathbb{R}^{n}$. Let, for each $j=1,2,\dots,t$
and $h=1,2,\dots,\ell_{j}$, $Q^{j}_{h}:\mathbb{R}^{m_{j}}\rightarrow2^{\mathbb{R}^{m_{j}}}$
be a finite family of point-to-set operators of $\mathbb{R}^{m_{j}}$
such that $Q^{j}_{h}(u)$ is nonempty, closed, and convex for each
$u\in\mathbb{R}^{m_{j}}$. Let $A_{j}:\mathbb{R}^{n}\rightarrow\mathbb{R}^{m_{j}}$
be a bounded linear operator for each $j=1,2,\dots,t$. Assume that
there exist positive constants $\{\eta_{i}\}^{r}_{i=1}$ and \soumitra{$\{\delta^{j}_{h}\}^{\ell_{j},t}_{h=1,j=1}$}
such that 
\begin{equation}
\|P_{C_{i}(x)}(z)-P_{C_{i}(y)}(z)\|\leq\eta_{i}\|x-y\|,\;\mathrm{for\;all\;}x,y,z\in\mathbb{R}^{n},i=1,2,\dots,r,\label{inq:1.1}
\end{equation}
and \soumitra{
\begin{eqnarray}
\|P_{Q^{j}_{h}(A_{j}x)}(A_{j}z)-P_{Q^{j}_{h}(A_{j}y)}(A_{j}z)\|\leq\delta^{j}_{h}\|A_{j}x-A_{j}y\|, & \mathrm{for\;all} & x,y,z\in\mathbb{R}^{n},h=1,2,\dots,\label{inq:1.2}\\
 &  & \ell_{j},j=1,2,\dots,t.\nonumber 
\end{eqnarray}
} Then the function $g$ is continuous on $\mathbb{R}^{n}$. \end{proposition}

\begin{proof} Using the assumptions and Proposition \ref{proposition3},
it follows that for each $i$, the function $\|x-P_{C_{i}(x)}(x)\|^{2}$
is continuous for any $x\in\mathbb{R}^{n}$. Similarly, for each $j,h$,
the function $\|A_{j}x-P_{Q^{j}_{h}(A_{j}x)}(A_{j}x)\|^{2}$ is continuous
for any $x\in\mathbb{R}^{n}$, since each $A_{j}$ is a bounded linear
operator. From the well-known fact that a finite linear combination
of continuous functions is continuous, it follows that the function
$g$ of (\ref{Zero=00003D00003D00003D00003D000020function}) is continuous
for all $x\in\mathbb{R}^{n}$.\end{proof}

\subsection{Algorithms for the Multi-Split Multiple-Sets Implicit Convex Feasibility
Problem}

\label{subsec:2} For each $x\in\mathbb{R}^{n}$, let the function
$g$ be of the form: 
\begin{align}
g(x):=u(x)+v(x),\label{Zero=00003D00003D00003D00003D000020function1}
\end{align}
where the summands are 
\begin{equation}
u(x):=\frac{1}{2}\sum^{r}_{i=1}\alpha_{i}\|x-P_{C_{i}(x)}(x)\|^{2}
\end{equation}
and 
\begin{equation}
v(x):=\frac{1}{2}\sum^{t}_{j=1}\sum^{\ell_{j}}_{h=1}\beta^{j}_{h}\|A_{j}x-P_{Q^{j}_{h}(A_{j}x)}(A_{j}x)\|^{2},
\end{equation}
and $\{\alpha_{i}\}^{r}_{i=1}$ and \soumitra{$\{\beta^{j}_{h}\}^{\ell_{j},t}_{h=1,j=1}$}
are positive constants.

From this point onward we assume that the point-to-set operators are
of the following form: 
\begin{equation}
\begin{cases}
\tag{A1}C_{i}(x):=\gamma_{i}U_{i}(\varOmega_{i})+M_{i}x\\
Q^{j}_{h}(A_{j}x):=\nu^{j}_{h}W^{j}_{h}(\varTheta^{j}_{h})+N^{j}_{h}A_{j}x,
\end{cases}\label{inq:3}
\end{equation}
where $x\in\mathbb{R}^{n}$, $M_{i}\in\mathbb{R}^{n\times n}$ and
$N^{j}_{h}\in\mathbb{R}^{m_{j}\times m_{j}}$ are matrices, $U_{i}\in\mathbb{R}^{n\times n},W^{j}_{h}\in\mathbb{R}^{m_{j}\times m_{j}}$
are unitary matrices, $\varOmega_{i}$ and $\varTheta^{j}_{h}$ are
nonempty, closed, and convex subsets of $\mathbb{R}^{n}$ and $\mathbb{R}^{m_{j}}$,
respectively, $\gamma_{i}>0$ and $\nu^{j}_{h}>0$ are fixed constants
for all $h,j,i$. In this section of the paper, we
assume that \eqref{inq:3} holds. These operators generalize
the point-to-set operators that are used in Lemma \ref{Lemma4}.

\begin{remark}\label{Rmk:1} 
\begin{enumerate}
\item Using Lemma \ref{Lemma4}, for $x\in\mathbb{R}^{n}$, it is straightforward
that 
\begin{equation}
u(x)=\frac{1}{2}\sum\limits^{r}_{i=1}\alpha_{i}\gamma^{2}_{i}\|(I-P_{\varOmega_{i}})(K_{i}x)\|^{2},
\end{equation}
where $K_{i}:=\frac{1}{\gamma_{i}}U^{T}_{i}(I-M_{i})$ for $i=1,2,\dots,r$.
\item From \cite[Proposition 3.1]{YCEN2016} it is true that the function
$u$ is convex, continuously differentiable on $\mathbb{R}^{n}$,
and its gradient is 
\begin{equation}
\nabla u(x)=\sum\limits^{r}_{i=1}\alpha_{i}\gamma^{2}_{i}K^{T}_{i}(I-P_{\varOmega_{i}})(K_{i}x),\;x\in\mathbb{R}^{n}.
\end{equation}
Furthermore, $\nabla u$ is Lipschitz continuous with constant 
\begin{equation}
L_{\nabla u}:=\sum\limits^{r}_{i=1}\alpha_{i}\|I-M_{i}\|^{2}_{2}.\label{lip:nabu}
\end{equation}
\soumitra{Indeed, using the fact that the projection
operator is firmly nonexpansive, the definition of matrix norm, and $\|K^{T}_{i}K_{i}\|_2=(1/\gamma^{2}_{i})\|I-M_{i}\|^{2}_{2}$,
for any $x,y\in\mathbb{R}^{n}$ we obtain {\allowdisplaybreaks 
\begin{eqnarray}
\|\nabla u(x) & - & \nabla u(y)\|\label{ineq1:Rmk:1}\nonumber \\
 & = & \Big\|\sum\limits^{r}_{i=1}\alpha_{i}\gamma^{2}_{i}K^{T}_{i}(I-P_{\varOmega_{i}})(K_{i}x)-\sum\limits^{r}_{i=1}\alpha_{i}\gamma^{2}_{i}K^{T}_{i}(I-P_{\varOmega_{i}})(K_{i}y)\Big\|\nonumber \\
 & \leq & \sum\limits^{r}_{i=1}\alpha_{i}\gamma^{2}_{i}\|K^{T}_{i}\|_{2}\|(I-P_{\varOmega_{i}})(K_{i}x)-(I-P_{\varOmega_{i}})(K_{i}y)\|\nonumber \\
 & \leq & \sum\limits^{r}_{i=1}\alpha_{i}\gamma^{2}_{i}\|K^{T}_{i}\|_{2}\|K_{i}x-K_{i}y\|\nonumber \\
 & \leq & \sum\limits^{r}_{i=1}\alpha_{i}\gamma^{2}_{i}\|K^{T}_{i}\|_{2}\|K_{i}\|_{2}\|x-y\|\nonumber \\
 & = & \sum\limits^{r}_{i=1}\alpha_{i}\gamma^{2}_{i}\|K_{i}\|^{2}_{2}\|x-y\|\nonumber \\
 & = & \sum\limits^{r}_{i=1}\alpha_{i}\gamma^{2}_{i}\|K^{T}_{i}K_{i}\|_{2}\|x-y\|\nonumber \\
 & = & \sum\limits^{r}_{i=1}\alpha_{i}\gamma^{2}_{i}\Big(\frac{1}{\gamma^{2}_{i}}\|I-M_{i}\|^{2}_{2}\Big)\|x-y\|\nonumber \\
 & = & \sum\limits^{r}_{i=1}\alpha_{i}\|I-M_{i}\|^{2}_{2}\|x-y\|.
\end{eqnarray}
} } 
\item \soumitra{The assumption $\eqref{inq:3}$ satisfies the continuity
assumption on \eqref{inq:1.1} and \eqref{inq:1.2}. Indeed, since
$C_{i}(x):=\gamma_{i}U_{i}(\varOmega_{i})+M_{i}x$ is a translation
of the fixed closed convex set $\gamma_{i}U_{i}(\varOmega_{i})$,
the nonexpansiveness of the projection operator and Lemma \ref{lemma:translation:projection}
imply that 
\begin{equation}
\|P_{C_{i}(x)}(z)-P_{C_{i}(y)}(z)\|\leq2\|M_{i}\|_{2}\|x-y\|,\;\mathrm{for\;all\;}x,y,z\in\mathbb{R}^{n},i=1,2,\dots,r.
\end{equation}
Hence, the condition $\eqref{inq:1.1}$ holds with $\eta_{i}=2\|M_{i}\|_{2}$.
Similarly, since $Q^{j}_{h}(A_{j}x):=\nu^{j}_{h}W^{j}_{h}(\varTheta^{j}_{h})+N^{j}_{h}A_{j}x$,
the nonexpansiveness of the projection operator and Lemma \ref{lemma:translation:projection}
imply that 
\begin{eqnarray}
\|P_{Q^{j}_{h}(A_{j}x)}(A_{j}z)-P_{Q^{j}_{h}(A_{j}y)}(A_{j}z)\|\leq2 & \|N^{j}_{h}\|_{2} & \|A_{j}x-A_{j}y\|,\mathrm{for\;all\;}x,y,z\in\mathbb{R}^{n},\\
 &  & h=1,2,\dots,\ell_{j},j=1,2,\dots,t.\nonumber 
\end{eqnarray}
Hence, the condition \eqref{inq:1.2} holds with $\delta^{j}_{h}=2\|N^{j}_{h}\|_{2}$.} 
\item \soumitra{In \eqref{inq:3}, we assumed that $Q^{j}_{h}$ is specified
only for Range$(A_{j})=\{A_{j}x\vert x\in\mathbb{R}^{n}\}$, i.e.,
\[
Q^{j}_{h}(A_{j}x)=\nu^{j}_{h}W^{j}_{h}(\varTheta^{j}_{h})+N^{j}_{h}A_{j}x,\;x\in\mathbb{R}^{n}.
\]
This is weaker than assuming that 
\[
Q^{j}_{h}(u)=\nu^{j}_{h}W^{j}_{h}(\varTheta^{j}_{h})+N^{j}_{h}u,\;u\in\mathbb{R}^{m_{j}},
\]
which defines $Q^{j}_{h}$ on the whole space $\mathbb{R}^{m_{j}}$.
However, when Range$(A_{j})=\mathbb{R}^{m_{j}}$, i.e., when $A_{j}$
is surjective, both assumptions are equivalent. } 
\end{enumerate}
\end{remark}

\begin{proposition}\label{Prop:1} The function $v$ is convex, continuously
differentiable on $\mathbb{R}^{n}$ and its gradient is given by the
formula 
\begin{equation}
\nabla v(x)=\sum^{t}_{j=1}\sum^{\ell_{j}}_{h=1}\beta^{j}_{h}{(\nu^{j}_{h})}^{2}{(L^{j}_{h})}^{T}(I-P_{\varTheta^{j}_{h}})(L^{j}_{h}x),\;x\in\mathbb{R}^{n}.\label{inq:5.1}
\end{equation}
Furthermore, $\nabla v$ is Lipschitz continuous with constant 
\begin{equation}
L_{\nabla v}:=\sum^{t}_{j=1}\sum^{\ell_{j}}_{h=1}\beta^{j}_{h}\|(I-N^{j}_{h})A_{j}\|^{2}_{2}.\label{lip:nab=00003D00003D00003D00003D000020v}
\end{equation}
\end{proposition}

\begin{proof} For any $x\in\mathbb{R}^{n}$, using \eqref{inq:3}
and Lemma \ref{Lemma4}, we obtain, {\allowdisplaybreaks 
\begin{eqnarray}
v(x) & = & \frac{1}{2}\sum^{t}_{j=1}\sum^{\ell_{j}}_{h=1}\beta^{j}_{h}\|A_{j}x-P_{Q^{j}_{h}(A_{j}x)}(A_{j}x)\|^{2}\label{inq:4}\nonumber \\
 & = & \frac{1}{2}\sum^{t}_{j=1}\sum^{\ell_{j}}_{h=1}\beta^{j}_{h}\big\| A_{j}x-P_{(\nu^{j}_{h}W^{j}_{h}(\varTheta^{j}_{h})+N^{j}_{h}A_{j}x})(A_{j}x)\big\|^{2}\nonumber \\
 & = & \frac{1}{2}\sum^{t}_{j=1}\sum^{\ell_{j}}_{h=1}\beta^{j}_{h}\Big\| A_{j}x-\nu^{j}_{h}W^{j}_{h}P_{\varTheta^{j}_{h}}\Big(\frac{1}{\nu^{j}_{h}}{(W^{j}_{h})}^{T}(A_{j}x-N^{j}_{h}A_{j}x)\Big)-N^{j}_{h}A_{j}x\Big\|^{2}\nonumber \\
 & = & \frac{1}{2}\sum^{t}_{j=1}\sum^{\ell_{j}}_{h=1}\beta^{j}_{h}\Big\|\nu^{j}_{h}W^{j}_{h}\Big(\frac{1}{\nu^{j}_{h}}{(W^{j}_{h})}^{T}(A_{j}x-N^{j}_{h}A_{j}x)-P_{\varTheta^{j}_{h}}\Big(\frac{1}{\nu^{j}_{h}}{(W^{j}_{h})}^{T}(A_{j}x-N^{j}_{h}A_{j}x)\Big)\Big)\Big\|^{2}\nonumber \\
 & = & \frac{1}{2}\sum^{t}_{j=1}\sum^{\ell_{j}}_{h=1}\beta^{j}_{h}{(\nu^{j}_{h})}^{2}\Big\|\Big(I-P_{\varTheta^{j}_{h}}\Big)\Big(\frac{1}{\nu^{j}_{h}}{(W^{j}_{h})}^{T}(A_{j}x-N^{j}_{h}A_{j}x)\Big)\Big\|^{2}\nonumber \\
 & = & \frac{1}{2}\sum^{t}_{j=1}\sum^{\ell_{j}}_{h=1}\beta^{j}_{h}{(\nu^{j}_{h})}^{2}\big\|(I-P_{\varTheta^{j}_{h}})(L^{j}_{h}x)\big\|^{2},
\end{eqnarray}
}where 
\begin{equation}
L^{j}_{h}:=\frac{1}{\nu^{j}_{h}}{(W^{j}_{h})}^{T}(I-N^{j}_{h})A_{j}.\label{ineq:extra}
\end{equation}

\soumitra{Clearly, $L^{j}_{h}\in\mathbb{R}^{m_{j}\times n}$. Using the definition of matrix norm, we get 
\begin{eqnarray}
\|L^{j}_{h}x-L^{j}_{h}y\|\leq\|L^{j}_{h}\|_{2}\|x-y\|,\label{ineq1::extra1}
\end{eqnarray}
for $x,y\in\mathbb{R}^{n}$. }

The convexity of $v$ follows from \eqref{inq:4}, see for instance
\cite[pp. 367--368]{EMAS2007}. Using the fact of \cite[Equation (3.15), p. 617]{YCEN2016},
we obtain 
\begin{equation}
\nabla v(x)=\sum^{t}_{j=1}\sum^{\ell_{j}}_{h=1}\beta^{j}_{h}{(\nu^{j}_{h})}^{2}{(L^{j}_{h})}^{T}(I-P_{\varTheta^{j}_{h}})(L^{j}_{h}x).\label{inq:5}
\end{equation}

Now, 
\begin{eqnarray}
(L^j_{h})^{T}L^{j}_{h} & = & \frac{1}{({\nu^{j}_{h})}^{2}}{({(W^{j}_{h})}^{T}(I-N^{j}_{h})A_{j})}^{T}{(W^{j}_{h})}^{T}(I-N^{j}_{h})A_{j}\label{inq:6}\nonumber \\
 & = & \frac{1}{({\nu^{j}_{h}})^{2}}((I-N^{j}_{h})A_{j})^{T}(I-N^{j}_{h})A_{j}.
\end{eqnarray}

Again, 
\begin{eqnarray}
\|(L^j_{h})^{T}L^{j}_{h}\|_{2} & =\frac{1}{{(\nu^{j}_{h})}^{2}}\|((I-N^{j}_{h})A_{j})^{T}(I-N^{j}_{h})A_{j}\|_{2}=\Big(\frac{1}{\nu^{j}_{h}}\|(I-N^{j}_{h})A_{j}\|_{2}\Big)^{2}.\label{inq:7.1}
\end{eqnarray}

Finally, for any $x,y\in\mathbb{R}^{n}$, using the fact that the projection operator is firmly nonexpansive, \soumitra{the definition of matrix norm,
\eqref{ineq1::extra1}, and \eqref{inq:7.1}, we obtain} {\allowdisplaybreaks
\begin{eqnarray}
\big\|\nabla v(x) & - & \nabla v(y)\big\|\nonumber \\
 & = & \Big\|\sum^{t}_{j=1}\sum^{\ell_{j}}_{h=1}\beta^{j}_{h}{(\nu^{j}_{h})}^{2}{(L^{j}_{h})}^{T}(I-P_{\varTheta^{j}_{h}})(L^{j}_{h}x)-\sum^{t}_{j=1}\sum^{\ell_{j}}_{h=1}\beta^{j}_{h}{(\nu^{j}_{h})}^{2}{(L^{j}_{h})}^{T}(I-P_{\varTheta^{j}_{h}})(L^{j}_{h}y)\Big\|\nonumber \\
 & = & \soumitra{\Big\|\sum^{t}_{j=1}\sum^{\ell_{j}}_{h=1}\beta^{j}_{h}{(\nu^{j}_{h})}^{2}{(L^{j}_{h})}^{T}((I-P_{\varTheta^{j}_{h}})(L^{j}_{h}x)-(I-P_{\varTheta^{j}_{h}})(L^{j}_{h}y))\Big\|}\nonumber \\
 & \leq & \sum^{t}_{j=1}\sum^{\ell_{j}}_{h=1}\beta^{j}_{h}{(\nu^{j}_{h})}^{2}\|{(L^{j}_{h})}^{T}\|_{2}\|L^{j}_{h}x-L^{j}_{h}y\|\nonumber \\
 & \leq & \sum^{t}_{j=1}\sum^{\ell_{j}}_{h=1}\beta^{j}_{h}{(\nu^{j}_{h})}^{2}\|{(L^{j}_{h})}^{T}\|_{2}\|L^{j}_{h}\|_{2}\|x-y\|\nonumber \\
 & = & \soumitra{\sum^{t}_{j=1}\sum^{\ell_{j}}_{h=1}\beta^{j}_{h}{(\nu^{j}_{h})}^{2}\|{L^{j}_{h}}\|^{2}_{2}\|x-y\|\nonumber}\\
 & = & \sum^{t}_{j=1}\sum^{\ell_{j}}_{h=1}\beta^{j}_{h}{(\nu^{j}_{h})}^{2}\|{(L^{j}_{h})}^{T}L^{j}_{h}\|_{2}\|x-y\|\nonumber \\
 & = & \sum^{t}_{j=1}\sum^{\ell_{j}}_{h=1}\beta^{j}_{h}{(\nu^{j}_{h})}^{2}\Big(\frac{1}{{(\nu^{j}_{h})}^{2}}\|(I-N^{j}_{h})A_{j}\|^{2}_{2}\Big)\|x-y\|\nonumber \\
 & = & \sum^{t}_{j=1}\sum^{\ell_{j}}_{h=1}\beta^{j}_{h}\|(I-N^{j}_{h})A_{j}\|^{2}_{2}\|x-y\|.
\end{eqnarray}
}

This implies that $\nabla v$ is $\sum^{t}_{j=1}\sum^{\ell_{j}}_{h=1}\beta^{j}_{h}\|(I-N^{j}_{h})A_{j}\|^{2}_{2}$-Lipschitz
continuous. \end{proof}

\begin{proposition}\label{Prop:2} The function $g$ of (\ref{Zero=00003D00003D00003D00003D000020function})
is convex and continuously differentiable on $\mathbb{R}^{n}$ and
its gradient is 
\begin{equation}
\nabla g(x)=\sum\limits^{r}_{i=1}\alpha_{i}\gamma^{2}_{i}K^{T}_{i}(I-P_{\varOmega_{i}})(K_{i}x)+\sum^{t}_{j=1}\sum^{\ell_{j}}_{h=1}\beta^{j}_{h}{(\nu^{j}_{h})}^{2}{(L^{j}_{h})}^{T}(I-P_{\varTheta^{j}_{h}})(L^{j}_{h}x),\;x\in\mathbb{R}^{n},
\end{equation}
where $\{\alpha_{i}\}^{r}_{i=1},\{\gamma_{i}\}^{r}_{i=1},\soumitra{\{\beta^{j}_{h}\}^{\ell_{j},t}_{h=1,j=1},\{\nu^{j}_{h}\}^{\ell_{j},t}_{h=1,j=1}}$
are positive constants. Furthermore, $\nabla g$ is $L_{\nabla g}:=\sum\limits^{r}_{i=1}\alpha_{i}\|I-M_{i}\|^{2}_{2}+\sum^{t}_{j=1}\sum^{\ell_{j}}_{h=1}\beta^{j}_{h}\|(I-N^{j}_{h})A_{j}\|^{2}_{2}$-Lipschitz
continuous. That is, 
\begin{equation}
\|\nabla g(x)-\nabla g(y)\|\leq L_{\nabla g}\|x-y\|,\;\mathrm{for\;all\;}x,y\in\mathbb{R}^{n}.
\end{equation}
\end{proposition}

\begin{proof} The function $g$ is convex and continuously differentiable
since the functions $u$ and $v$ are both convex and continuously
differentiable on $\mathbb{R}^{n}$ (see Remark \ref{Rmk:1} and Proposition
\ref{Prop:1}).

Since for any $x\in\mathbb{R}^{n}$, $g(x)=u(x)+v(x)$, it follows
that $\nabla g(x)=\nabla u(x)+\nabla v(x)$. Thereby, for any $x,y\in\mathbb{R}^{n}$
we have \soumitra{
\begin{eqnarray}
\|\nabla g(x)-\nabla g(y)\| & = & \|\nabla u(x)+\nabla v(x)-\nabla u(y)-\nabla v(y)\|\label{inq:6}\nonumber \\
 & \leq & \|\nabla u(x)-\nabla u(y)\|+\|\nabla v(x)-\nabla v(y)\|.
\end{eqnarray}
}

Since $\nabla u$ is $\sum\limits^{r}_{i=1}\alpha_{i}\|I-M_{i}\|^{2}_{2}$-Lipschitz
continuous (see Remark \ref{Rmk:1}) and $\nabla v$ is\\
 $\sum^{t}_{j=1}\sum^{\ell_{j}}_{h=1}\beta^{j}_{h}\|(I-N^{j}_{h})A_{j}\|^{2}_{2}$-Lipschitz
continuous (see Proposition \ref{Prop:1}), from inequality \eqref{inq:6}
we obtain \soumitra{
\begin{eqnarray}
 & \| & \nabla g(x)-\nabla g(y)\|\label{inq:7}\nonumber \\
 & \leq & \|\nabla u(x)-\nabla u(y)\|+\|\nabla v(x)-\nabla v(y)\|\nonumber \\
 & \leq & \Big(\sum\limits^{r}_{i=1}\alpha_{i}\|I-M_{i}\|^{2}_{2}+\sum^{t}_{j=1}\sum^{\ell_{j}}_{h=1}\beta^{j}_{h}\|(I-N^{j}_{h})A_{j}\|^{2}_{2}\Big)\|x-y\|,\;\mathrm{for\;all\;}x,y\in\mathbb{R}^{n}.
\end{eqnarray}
} Therefore, $\nabla g$ is $\sum\limits^{r}_{i=1}\alpha_{i}\|I-M_{i}\|^{2}_{2}+\sum^{t}_{j=1}\sum^{\ell_{j}}_{h=1}\beta^{j}_{h}\|(I-N^{j}_{h})A_{j}\|^{2}_{2}$-Lipschitz
continuous. \end{proof}

\subsubsection{The simultaneous algorithm and its convergence}

\label{subsec:The-simultaneous-algorithm}

In this subsection, we introduce a new simultaneous (parallel) projection
algorithm to solve the problem \eqref{MSMSICFP} for point-to-set
operators of the form as in \eqref{inq:3}.

 \makeatletter \newenvironment{breakablealgorithm}{
    \begin{center}
       \refstepcounter{algorithm}
        \renewcommand{\caption}[1]
        {
           \addcontentsline{loa}{algorithm}{\protect\numberline{\thealgorithm}##1}
           \parbox{\textwidth}
            {
                \hrule height.8pt depth0pt \kern2pt
                {\raggedright\textbf{\fname@algorithm~\thealgorithm} ##1\par}
                \kern2pt\hrule\kern2pt
            }
        }
}{
        \kern2pt\hrule\relax
    \end{center}
}\begin{breakablealgorithm} \caption{\soumitra{A Simultaneous Projection Gradient Algorithm (SPGA)} for
MSMS-ICFP}
\label{Algo:SPGA} \begin{algorithmic}[1] \vspace{-0.2cm}
 \State \textbf{Input}: \soumitra{$x^{0}\in\mathbb{R}^{n}$.}
\State Compute $L_{\nabla g}=\sum\limits^{r}_{i=1}\alpha_{i}\|I-M_{i}\|^{2}_{2}+\sum^{t}_{j=1}\sum^{\ell_{j}}_{h=1}\beta^{j}_{h}\|(I-N^{j}_{h})A_{j}\|^{2}_{2}$.
\For{$i=1,2,\dots,r$} 
\State Compute 
\begin{equation}
K_{i}=\frac{1}{\gamma_{i}}U^{T}_{i}(I-M_{i}).
\end{equation}
\EndFor 
\For{$j=1,2,\dots,t$} 
\For{$h=1,2,\dots,\ell_{j}$}
\State Compute 
\begin{eqnarray}
\begin{aligned}L^{j}_{h}=\frac{1}{\nu^{j}_{h}}{(W^{j}_{h})}^{T}(I-N^{j}_{h})A_{j}.\end{aligned}
\end{eqnarray}
\EndFor 
\EndFor 
\State Choose $\lambda\in(0,2/L_{\nabla g})$.
\soumitra{\For{$k=0,1,2,\dots$} \State Compute 
\begin{eqnarray}
{\small \begin{aligned}x^{k+1} & :=x^{k}-\lambda\sum\limits^{r}_{i=1}\alpha_{i}\gamma^{2}_{i}K^{T}_{i}(I-P_{\varOmega_{i}})(K_{i}x^{k})\\
 & \hspace{1.5cm}-\lambda\sum^{t}_{j=1}\sum^{\ell_{j}}_{h=1}\beta^{j}_{h}{(\nu^{j}_{h})}^{2}{(L^{j}_{h})}^{T}(I-P_{\varTheta^{j}_{h}})(L^{j}_{h}x^{k}).
\end{aligned}
}\label{Algo1:ineq:8}
\end{eqnarray}
} \EndFor \end{algorithmic} \end{breakablealgorithm}

\medskip{}
\begin{remark} In Algorithm \ref{Algo:SPGA}, we use a constant step-size
and prove the convergence of the sequence $\{x^{k}\}^{\infty}_{k=0}$
generated by it in Theorem \ref{MainThm:1} below. However, our proof
of convergence carries over also with an appropriate choice of variable
step-sizes. \end{remark}

\begin{theorem}\label{MainThm:1} Denote the solution set of the
problem \eqref{MSMSICFP} by $\varTheta$ and assume that it is nonempty.
Then any sequence $\{x^{k}\}^{\infty}_{k=0}$ generated by Algorithm
\ref{Algo:SPGA} converges in norm to a point in $\varTheta$. \end{theorem}

\begin{proof} \soumitra{Since $\varTheta$ is nonempty, the problem
\eqref{MSMSICFP} is consistent. Therefore, by Proposition \ref{proposition2}
and Corollary \ref{Cor:1}, the solution set of \eqref{MSMSICFP}
coincides with the global minimizers of the unconstrained minimization
problem \eqref{minimization} and the minimum value is zero. Moreover,
by Proposition \ref{Prop:2}, the function $g$ of (\ref{Zero=00003D00003D00003D00003D000020function})
is convex, continuously differentiable, and its gradient is $L_{\nabla g}$-Lipschitz
continuous. 
Since the step-size $\lambda\in(0,2/L_{\nabla g})$, all the assumptions
of Corollary \ref{Co1:Grad:thm} are satisfied. By the expression
in \eqref{Algo1:ineq:8}, our Algorithm \ref{Algo:SPGA} is a standard
gradient descent algorithm for the problem \eqref{MSMSICFP}. Therefore,
by applying Corollary \ref{Co1:Grad:thm}, we conclude that the sequence
$\{x^{k}\}^{\infty}_{k=0}$ generated by Algorithm \ref{Algo:SPGA}
converges in norm to a point $y^{*}\in\arg\min_{x\in\mathbb{R}^{n}}g(x)$.
Since the set of all global minimizers of $g$ coincides with $\varTheta$,
$y^{*}\in\varTheta$. Thus, the sequence $\{x^{k}\}_{k=0}^\infty$ converges
in norm to a solution of the problem \eqref{MSMSICFP}}. \end{proof}


\begin{remark} The above analysis holds for a special type of point-to-set
operators. Therefore, it is natural to ask whether it can be extended
for other or more general point-to-set operators? \end{remark}


\subsection{Bregman Projections and A Generalized Proximity Function}

Next we generalize the proximity function $g$ of \eqref{Zero=00003D00003D00003D00003D000020function}
or \eqref{Zero=00003D00003D00003D00003D000020function1} by replacing
 the orthogonal projections by Bregman projections. Bregman distances
and Bregman projections enjoy a large body of published works in the
mathematical literature and have been used in diverse fields such
as statistics, machine learning, information geometry and optimization,
to name but a few.

\soumitra{We denote by $\mathscr{B}(\mathbb{R}^{n})$ the set of
all Bregman functions $f:\mathbb{R}^{n}\rightarrow\mathbb{R}$ with
{\em zone} $\mathbb{R}^{n}$. We refer the reader to \cite[Chapter 2]{YCEN1997}
for a precise definition of such functions.}

For $f\in\mathscr{B}(\mathbb{R}^{n})$ the Bregman distance is given
by 
\begin{eqnarray}
\begin{aligned}\label{ineq:bregmain}D_{f}(y,x):=f(y)-f(x)-\left\langle \nabla f(x),y-x\right\rangle .\end{aligned}
\end{eqnarray}

Consider a nonempty, closed, and convex subset $C$ of $\mathbb{R}^{n}$.
Then the Bregman projection of $x$ onto $C$ with respect to $f\in\mathscr{B}(\mathbb{R}^{n})$
is given by 
\begin{eqnarray}
\begin{aligned}P^{f}_{C}(x):=\arg\min_{y\in C}D_{f}(y,x).\end{aligned}
\label{Breg:projection}
\end{eqnarray}

It is well-known that the Bregman projection exists, and that it is
unique when $f\in\mathscr{B}(\mathbb{R}^{n})$. We recall the following result for Bregman projections which is \cite[Theorem 2.4.2]{YCEN1997}
specialized for $f\in\mathscr{B}(\mathbb{R}^{n})$.

\begin{theorem}\label{Breg:Proj:Characterization} Let $f\in\mathscr{B}(\mathbb{R}^{n})$.
Then for any $x\in\mathbb{R}^{n}$, $x^{*}=P^{f}_{C}(x)$ if and only
if 
\begin{eqnarray}
\begin{aligned}\left\langle \nabla f(x)-\nabla f(x^{*}),y-x^{*}\right\rangle \leq0,\;\text{for all}\;y\in C.\end{aligned}
\end{eqnarray}
\end{theorem}

Now 
\begin{eqnarray}
\begin{aligned}D_{f}(P^{f}_{C}(x),x)=f(P^{f}_{C}(x))-f(x)-\left\langle \nabla f(x),P^{f}_{C}(x)-x\right\rangle .\end{aligned}
\label{Breg:ineq1}
\end{eqnarray}

\soumitra{ \begin{lemma}\label{lem1:Breg:gradient} Let $f\in\mathscr{B}(\mathbb{R}^{n})$,
and let $C$ be a nonempty, closed, and convex subset of $\mathbb{R}^{n}$.
Suppose that $f$ is twice continuously differentiable, that $D_{f}(P^{f}_{C}(\cdot),\cdot)$
is a convex function, and that for every hyperplane $H$ the derivatives
\begin{eqnarray}
\frac{\partial}{\partial x_{q}}(P^{f}_{H}(x))_{p},\;\text{for}\;1\leq p,q\leq n,\label{ineq1:Breg:gradient}
\end{eqnarray}
exist and are continuous for all $x\in\mathbb{R}^{n}$, then $D_{f}(P^{f}_{C}(\cdot),\cdot)$
is a continuously differentiable function and 
\begin{eqnarray}
\begin{aligned}\label{Breg:gradient}\nabla D_{f}(P^{f}_{C}(x),x)=\nabla^{2}f(x)(x-P^{f}_{C}(x)),\;\text{for all}\;x\in\mathbb{R}^{n}.\end{aligned}
\end{eqnarray}
\end{lemma} \begin{proof} Substitute $S=\mathbb{R}^{n}$, $\varOmega=C$,
and $d^{f}_{C}(\cdot)=D_{f}(P^{f}_{C}(\cdot),\cdot)$ in \cite[Proposition 12]{YCEN2004}.
Then $C\cap\text{cl}(S)\neq\varnothing$ since $C\neq\varnothing$. Moreover,
it is clear that $f$ is zone consistent with respect to $C$, that
is, $x\in\mathbb{R}^{n}$ implies $P^{f}_{C}(x)\in\mathbb{R}^{n}$.
Thus, all the conditions of \cite[Proposition 12]{YCEN2004} are satisfied
and the result follows. \end{proof} }

\soumitra{ \begin{remark} Let $f\in\mathscr{B}(\mathbb{R}^{n})$,
and let $C$ be a nonempty, closed, and convex subset of $\mathbb{R}^{n}$.
Then $D_{f}(P^{f}_{C}(\cdot),\cdot)$ may not be a convex function.
For example, let $C=\{0\}\subset\mathbb{R}$ and $f(x)=(x^{2}/60)+e^{-x}$.
It is clear that $f$ is a strongly convex function with constant
$1/30$. Hence, $D_{f}(y,x)\geq\parallel y-x\parallel{}^{2}/60$ ,
for any $x,y\in\mathbb{R}$. Therefore, $f$ is a Bregman function
with zone $\mathbb{R}$. Since $C=\{0\}$, $P^{f}_{C}(x)=0$ for all
$x\in\mathbb{R}$. Thus, $D_{f}(0,x)=f(0)-f(x)+xf^{\prime}(x)=1+(x^{2}/60)-(1+x)e^{-x}$.
Now, $D^{\prime}_{f}(0,x)=(x/30)+xe^{-x}$ and $D^{\prime\prime}_{f}(0,x)=(1/30)+(1-x)e^{-x}$.
This implies that $D^{\prime\prime}_{f}(0,2)=(1/30)-e^{-2}=-0.1020<0$.
As a result, $D_{f}(P^{f}_{C}(\cdot),\cdot)$ is not convex. \end{remark}
}

\soumitra{Unless otherwise stated, by $g_{j}\in\mathscr{B}(\mathbb{R}^{m_{j}})$,
we mean that $g_{j}:\mathbb{R}^{m_{j}}\rightarrow\mathbb{R}$ is a
Bregman function with zone $\mathbb{R}^{m_{j}}$, $j=1,2,\dots,t$.
We denote $g:=(g_{j})^{t}_{j=1}=(g_{1},g_{2},\dots,g_{t}).$}

\subsubsection{Our generalized proximity function}

\label{subsec:Our-Bregman-proximity}

Let $f\in\mathscr{B}(\mathbb{R}^{n})$ and $g_{j}\in\mathscr{B}(\mathbb{R}^{m_{j}})$
for $j=1,2,\dots,t$. By replacing the orthogonal projections by Bregman
projections we define the function $p_{f,g}:\mathbb{R}^{n}\rightarrow\mathbb{R}$
by 
\begin{eqnarray}
\begin{aligned}p_{f,g}(x):=\sum^{r}_{i=1}\alpha_{i}D_{f}(P^{f}_{C_{i}(x)}(x),x)+\sum^{t}_{j=1}\soumitra{\sum^{\ell_{j}}_{h=1}}\beta^{j}_{h}D_{g_{j}}\bigg(P^{g_{j}}_{Q^{j}_{h}(A_{j}x)}(A_{j}x),A_{j}x\bigg),\;x\in\mathbb{R}^{n},\end{aligned}
\label{Breg:ineq2}
\end{eqnarray}
where $\alpha_{i}>0,\beta^{j}_{h}>0$ are parameters, and $C_{i},Q^{j}_{h}$
are the point-to-set operators defined above for $h,j,i$.
This is our generalized proximity function.

\begin{remark} \soumitra{It is important to note that, in general,
a single Bregman function cannot be used to define the relevant Bregman
distances in all spaces arising in a split problem. Indeed, the image
of $x\in\mathbb{R}^{n}$ under the bounded linear operator $A_{j}$
belongs to $\mathbb{R}^{m_{j}}$ for $j=1,2,\dots,t$. Due to this
restriction, we employ a Bregman function $f\in\mathscr{B}(\mathbb{R}^{n})$
for nonempty, closed, and convex sets $C_{i}$ in the domain space
and, for each $j=1,2,\dots,t$, a Bregman function $g_{j}\in\mathscr{B}(\mathbb{R}^{m_{j}})$
for nonempty, closed, and convex sets $Q^{j}_{h}$ for all $h=1,2,\dots,\ell_{j}$,
in the corresponding image space. Thus, the Bregman distances and
the Bregman projections are defined in the appropriate underlying
Euclidean spaces.} \end{remark}

\begin{proposition}\label{Breg:proposition1} \soumitra{Let $f\in\mathscr{B}(\mathbb{R}^{n})$
and $g_{j}\in\mathscr{B}(\mathbb{R}^{m_{j}})$ for $j=1,2,\dots,t$.}
The function $p_{f,g}:\mathbb{R}^{n}\rightarrow\mathbb{R}$ of \eqref{Breg:ineq2}
is a generalized proximity function for the problem \eqref{MSMSICFP}.
\end{proposition}

\begin{proof} By definition $p_{f,g}(x)$ is non-negative for all
$x\in\mathbb{R}^{n}$. It remains to show that the function $p_{f,g}$
satisfies the property $(P_{2})$ of Definition \ref{Def:1}. Indeed,
let $x^{*}\in\mathbb{R}^{n}$ be a solution of the problem \eqref{MSMSICFP}.
Then using the Bregman projection characterization it is straightforward
to verify that 
\begin{eqnarray}
\begin{aligned}\label{Breg:por1:ineq1}x^{*}=P^{f}_{C_{i}(x^{*})}(x^{*})\text{ and }A_{j}x^{*}=P^{g_{j}}_{Q^{j}_{h}(A_{j}x^{*})}(A_{j}x^{*}),\;\mathrm{for\;all\;}h,j,i.\end{aligned}
\end{eqnarray}

This implies that $p_{f,g}(x^{*})=0$.

Conversely, suppose $x^{*}\in\mathbb{R}^{n}$ solves the equation
$p_{f,g}(x)=0$. This implies that 
\begin{eqnarray}
\begin{aligned}\sum^{r}_{i=1}\alpha_{i}D_{f}(P^{f}_{C_{i}(x^{*})}(x^{*}),x^{*})+\sum^{t}_{j=1}\soumitra{\sum^{\ell_{j}}_{h=1}}\beta^{j}_{h}D_{g_{j}}\bigg(P^{g_{j}}_{Q^{j}_{h}(A_{j}x^{*})}(A_{j}x^{*}),A_{j}x^{*}\bigg)=0.\end{aligned}
\label{Breg:pro1:ineq2}
\end{eqnarray}

\begin{eqnarray}
\begin{aligned}\sum^{r}_{i=1}\alpha_{i}D_{f}(P^{f}_{C_{i}(x^{*})}(x^{*}),x^{*})=0\;\text{and}\;\sum^{t}_{j=1}\soumitra{\sum^{\ell_{j}}_{h=1}}\beta^{j}_{h}D_{g_{j}}\bigg(P^{g_{j}}_{Q^{j}_{h}(A_{j}x^{*})}(A_{j}x^{*}),A_{j}x^{*}\bigg)=0.\end{aligned}
\label{Breg:pro1:ineq3}
\end{eqnarray}

Consequently, \soumitra{
\begin{eqnarray}
\begin{aligned} & \alpha_{i}D_{f}(P^{f}_{C_{i}(x^{*})}(x^{*}),x^{*})=0,\;\mathrm{for\;all\;}i=1,2,\dots,r,\;\text{and}\\
 & \sum^{\ell_{j}}_{h=1}\beta^{j}_{h}D_{g_{j}}\Big(P^{g_{j}}_{Q^{j}_{h}(A_{j}x^{*})}(A_{j}x^{*}),A_{j}x^{*}\Big)=0,\;\mathrm{for\;all\;}j=1,2,\dots,t.
\end{aligned}
\label{Breg:pro1:ineq4}
\end{eqnarray}
}

Therefore, 
\begin{eqnarray}
\begin{aligned}x^{*}\in\bigcap^{r}_{i=1}C_{i}(x^{*})\text{ and }A_{j}x^{*}\in\bigcap^{\ell_{j}}_{h=1}Q^{j}_{h}(A_{j}x^{*}),\;\text{for each }j=1,2,\dots,t.\end{aligned}
\label{Breg:pro1:ineq5}
\end{eqnarray}
Thus, $x^{*}$ is a solution of the problem \eqref{MSMSICFP}.\end{proof}

\begin{remark}\label{Breg:rmk} The analysis presented in the previous
sections cannot be directly extended to the general case in the Bregman
setting, since it is not well established whether Lemma \ref{Lemma4}
holds for Bregman projections. However, we do make some progress for
a special case in the next subsection.\end{remark}

Various types of Bregman distances and proximity functions were considered
in different contexts in the literature. See, e.g., \cite{CBYR2001,HHBA2006,RSBU2021}.

\bigskip{}

\subsubsection{The split feasibility problem with multiple output sets}

\label{subsec:The-split-feasibility}

An important special case of \eqref{MSMSICFP} is \label{Sec:splcas}
the split feasibility problem with multiple output sets, which was
considered in \cite[Subsection 4.1]{SREI2022}. We derive here some
generalized results about this case by using the generalized proximity
function.

Let for all $x\in\mathbb{R}^{n}$, the point-to-set operator $C_{i}(x)=C_{i}$,
$i=1,2,\dots,r$, be a nonempty, closed, and convex subset of $\mathbb{R}^{n}$.
Let for all $y\in\mathbb{R}^{m_{j}}$, the point-to-set operator $Q^{j}_{h}(y)=Q^{j}_{h}$
be a nonempty, closed, and convex subset of $\mathbb{R}^{m_{j}}$
for $h=1,2,\dots,\ell_{j}$, $j=1,2,\dots,t$. Let $A_{j}:\mathbb{R}^{n}\rightarrow\mathbb{R}^{m_{j}}$
be a bounded linear operator for each $j=1,2,\dots,t$. Then the problem
\eqref{MSMSICFP} consists of finding a point $x^{*}\in\mathbb{R}^{n}$
such that 
\begin{align}
x^{*}\in\bigcap\limits^{r}_{i=1}C_{i}\;\text{ and }\;A_{j}x^{*}\in\bigcap\limits^{\ell_{j}}_{h=1}Q^{j}_{h},\;\text{for each }j=1,2,\dots,t.\tag{SFPMOS}\label{SFPMOS}
\end{align}
The problem \eqref{SFPMOS} was studied in \cite[Subsection 4.1]{SREI2022}
for infinite-dimensional Hilbert spaces.

\begin{corollary}\label{Bregman:proposition2} \soumitra{The function
$p:\mathbb{R}^{n}\rightarrow\mathbb{R}$, defined by 
\begin{equation}
p(x):=\frac{1}{2}\sum^{r}_{i=1}\alpha_{i}\|x-P_{C_{i}}(x)\|^{2}+\frac{1}{2}\sum^{t}_{j=1}\sum^{\ell_{j}}_{h=1}\beta^{j}_{h}\|A_{j}x-P_{Q^{j}_{h}}(A_{j}x)\|^{2},\;x\in\mathbb{R}^{n},\label{ineq1:Bregman:proposition2:Zerofunction}
\end{equation}
is a proximity function for the problem \eqref{SFPMOS}. } \end{corollary}

\begin{proof} \soumitra{The proof is analogous to Proposition
\ref{proposition2}.} \end{proof}

\medskip{}
 Let $f\in\mathscr{B}(\mathbb{R}^{n})$ and $g_{j}\in\mathscr{B}(\mathbb{R}^{m_{j}})$
for $j=1,2,\dots,t$. Define a function $p_{f,g}:\mathbb{R}^{n}\rightarrow\mathbb{R}$
by 
\begin{eqnarray}
\begin{aligned}p_{f,g}(x):=\sum^{r}_{i=1}\alpha_{i}D_{f}(P^{f}_{C_{i}}(x),x)+\sum^{t}_{j=1}\soumitra{\sum^{\ell_{j}}_{h=1}}\beta^{j}_{h}D_{g_{j}}\bigg(P^{g_{j}}_{Q^{j}_{h}}(A_{j}x),A_{j}x\bigg),\;x\in\mathbb{R}^{n}.\end{aligned}
\label{Breg:SFPMO:ineq1}
\end{eqnarray}

We have the following result.

\begin{corollary}\label{prox:SFPMO} \soumitra{Let $f\in\mathscr{B}(\mathbb{R}^{n})$
and $g_{j}\in\mathscr{B}(\mathbb{R}^{m_{j}})$ for $j=1,2,\dots,t$.
The function $p_{f,g}:\mathbb{R}^{n}\rightarrow\mathbb{R}$ defined
above in \eqref{Breg:SFPMO:ineq1} is a generalized proximity function
for the problem \eqref{SFPMOS}.} \end{corollary}

\begin{proof} The proof is analogous to Proposition \ref{Breg:proposition1}.
\end{proof}

\begin{lemma}\label{Breg:lem:BregProj} \soumitra{Let $f\in\mathscr{B}(\mathbb{R}^{n})$
and $C$ be a nonempty, closed, and convex subset of $\mathbb{R}^{n}$.
Additionally, assume that the function $f$ is $\sigma$-strongly
convex and $\nabla f$ is $L$-Lipschitz continuous. Then 
\begin{eqnarray}
\begin{aligned}\|P^{f}_{C}(x)-P^{f}_{C}(y)\|\leq\frac{L}{\sigma}\|x-y\|,\;\text{for all}\;x,y\in\mathbb{R}^{n}.\end{aligned}
\end{eqnarray}
} \end{lemma}

\begin{proof} 
\soumitra{Since $f\in\mathscr{B}(\mathbb{R}^{n})$, $P^{f}_{C}(x)$
exists and is unique for each $x\in\mathbb{R}^{n}$. For $x,y\in\mathbb{R}^{n}$,
by Theorem \ref{Breg:Proj:Characterization}, we get 
\begin{eqnarray}
\left\langle \nabla f(x)-\nabla f(P^{f}_{C}(x)),P^{f}_{C}(y)-P^{f}_{C}(x)\right\rangle \leq0\label{ineq1:Breg:Proj:Characterization}
\end{eqnarray}
and 
\begin{eqnarray}
\left\langle \nabla f(y)-\nabla f(P^{f}_{C}(y)),P^{f}_{C}(x)-P^{f}_{C}(y)\right\rangle \leq0.\label{ineq2:Breg:Proj:Characterization}
\end{eqnarray}
Adding \eqref{ineq1:Breg:Proj:Characterization} and \eqref{ineq2:Breg:Proj:Characterization},
we obtain 
\begin{eqnarray}
\left\langle \nabla f(P^{f}_{C}(y))-\nabla f(P^{f}_{C}(x)),P^{f}_{C}(y)-P^{f}_{C}(x)\right\rangle \leq\left\langle \nabla f(y)-\nabla f(x),P^{f}_{C}(y)-P^{f}_{C}(x)\right\rangle .\label{ineq3:Breg:Proj:Characterization}
\end{eqnarray}
Using the facts that $f$ is continuously differentiable and $\sigma$-strongly
convex, and $\nabla f$ is $L$-Lipschitz continuous, from \eqref{ineq3:Breg:Proj:Characterization}
we get 
\begin{eqnarray}
\sigma\|P^{f}_{C}(y)-P^{f}_{C}(x)\|^{2}\leq L\|y-x\|\|P^{f}_{C}(y)-P^{f}_{C}(x)\|.\label{ineq4:Breg:Proj:Characterization}
\end{eqnarray}
If $P^{f}_{C}(y)=P^{f}_{C}(x)$, the result is immediate. Otherwise,
dividing by $\|P^{f}_{C}(y)-P^{f}_{C}(x)\|$ both sides of \eqref{ineq4:Breg:Proj:Characterization},
we obtain the desired inequality.} \end{proof}

\begin{corollary}\label{Breg:cor1:gradient} \soumitra{Let $f\in\mathscr{B}(\mathbb{R}^{n})$
and $g_{j}\in\mathscr{B}(\mathbb{R}^{m_{j}})$ for $j=1,2,\dots,t$.
Suppose, for each $i=1,2,\dots,r$, the pair $(f,C_{i})$ satisfies
the hypotheses of Lemma \ref{lem1:Breg:gradient}. Additionally, assume
that for each $j=1,2,\dots,t$ and $h=1,2,\dots,\ell_{j}$, the pair
$(g_{j},Q^{j}_{h})$ satisfies the hypotheses of Lemma \ref{lem1:Breg:gradient}.
Then the function defined in \eqref{Breg:SFPMO:ineq1} is continuously
differentiable and for every $x\in\mathbb{R}^{n}$, its gradient is
given by 
\begin{eqnarray}
\begin{aligned}\nabla p_{f,g}(x)=\sum^{r}_{i=1}\alpha_{i}\nabla^{2}f(x)(x-P^{f}_{C_{i}}(x))+\sum^{t}_{j=1}\soumitra{\sum^{\ell_{j}}_{h=1}}\beta^{j}_{h}A^{T}_{j}\nabla^{2}g_{j}(A_{j}x)(A_{j}x-P^{g_{j}}_{Q^{j}_{h}}(A_{j}x)),\end{aligned}
\label{ineq1:Breg:cor1:gradient}
\end{eqnarray}
where $\nabla^{2}f$ stands for the Hessian matrix of $f$ and $\nabla^{2}g_{j}$
stands for the Hessian matrix of $g_{j}$. } \end{corollary}

\begin{proof} \soumitra{For each $i=1,2,\dots,r$, the pair $(f,C_{i})$
satisfies the hypotheses of Lemma \ref{lem1:Breg:gradient}. Therefore,
using Lemma \ref{lem1:Breg:gradient}, we conclude that for each $i=1,2,\dots,r$,
$D_{f}(P^{f}_{C_{i}}(\cdot),\cdot)$ is a continuously differentiable
function and 
\begin{eqnarray}
\begin{aligned}\label{ineq2:Breg:cor1:gradient}\nabla D_{f}(P^{f}_{C_{i}}(x),x)=\nabla^{2}f(x)(x-P^{f}_{C_{i}}(x)),\;\text{for all}\;x\in\mathbb{R}^{n}.\end{aligned}
\end{eqnarray}
}

\soumitra{Similarly, for each $j=1,2,\dots,t$ and $h=1,2,\dots,\ell_{j}$,
the pair $(g_{j},Q^{j}_{h})$ satisfies the hypotheses of Lemma \ref{lem1:Breg:gradient}.
Therefore, using Lemma \ref{lem1:Breg:gradient}, we conclude that
for each $j=1,2,\dots,t$ and $h=1,2,\dots,\ell_{j}$, $D_{g_{j}}(P^{g_{j}}_{Q^{j}_{h}}(\cdot),\cdot)$
is a continuously differentiable function and 
\begin{eqnarray}
\begin{aligned}\label{ineq3:Breg:cor1:gradient}\nabla D_{g_{j}}(P^{g_{j}}_{Q^{j}_{h}}(y),y)=\nabla^{2}g_{j}(y)(y-P^{g_{j}}_{Q^{j}_{h}}(y)),\;\text{for all}\;y\in\mathbb{R}^{m_{j}}.\end{aligned}
\end{eqnarray}
}

\soumitra{Now, for each $j=1,2,\dots,t$ and $h=1,2,\dots,\ell_{j}$,
taking $y=A_{j}x$ in \eqref{ineq3:Breg:cor1:gradient} and applying
the chain rule, we obtain 
\begin{eqnarray}
\begin{aligned}\label{ineq4:Breg:cor1:gradient}\nabla D_{g_{j}}(P^{g_{j}}_{Q^{j}_{h}}(A_{j}x),A_{j}x)=A^{T}_{j}\nabla^{2}g_{j}(A_{j}x)(A_{j}x-P^{g_{j}}_{Q^{j}_{h}}(A_{j}x)),\;\text{for all}\;x\in\mathbb{R}^{n}.\end{aligned}
\end{eqnarray}
} \end{proof}

\begin{lemma}\label{Lemma:Breg:Diff3} \soumitra{Assume that all the hypotheses of Corollary \ref{Breg:cor1:gradient} hold.
Assume that for each $i=1,2,\dots,r,$ there exists a constant $L_1^{i}>0$ such that 
\begin{eqnarray}\label{ineq1:Lemma:Breg:Diff3}
\begin{aligned}
x\rightarrow\nabla^{2}f(x)(x-P_{C_{i}}^{f}(x))
\end{aligned}
\end{eqnarray}
is $L_1^i$-Lipschitz continuous.
Similarly, assume that
for each $j=1,2,\dots,t$ and $h=1,2,\dots,\ell_j$, there exists a constant $L_{2}^{j,h}>0$ such that 
\begin{eqnarray}\label{ineq2:Lemma:Breg:Diff3}
\begin{aligned}
    x\rightarrow\nabla^{2}g_j(A_{j}x)(A_{j}x-P_{Q_{h}^{j}}^{g_j}(A_{j}x))
\end{aligned}
\end{eqnarray}
is $L_{2}^{j,h}$-Lipschitz continuous.
Then $\nabla p_{f,g}$ is $\sum_{i=1}^{r}\alpha_{i}L_1^i+\sum_{j=1}^{t}\sum_{h=1}^{\ell_{j}}\beta_{h}^{j}L_2^{j,h}\|A_{j}\|_{2}$-Lipschitz continuous.}
\end{lemma}

\begin{proof} 
\sloppy
\soumitra{The proof directly follows from Corollary \ref{Breg:cor1:gradient} and using the facts that 
$\sum_{i=1}^{r}\alpha_{i}\nabla^{2}f(x)(x-P_{C_{i}}^{f}(x))$
is $\sum_{i=1}^{r}\alpha_{i}L_1^i$-Lipschitz
continuous and $\sum_{j=1}^{t}\sum_{h=1}^{\ell_{j}}\beta_{h}^{j}A_{j}^{T}\nabla^{2}g_j(A_{j}x)(A_{j}x-P_{Q_{h}^{j}}^{g_j}(A_{j}x))$
is Lipschitz continuous with constant 
\begin{eqnarray}
\begin{aligned}\sum_{j=1}^{t}\sum_{h=1}^{\ell_{j}}\beta_{h}^{j}L_2^{j,h}\|A_{j}\|_{2}.\nonumber
\end{aligned}
\end{eqnarray}}
\end{proof}

\begin{remark}\label{remark1:1:1:Breg:lem:BregProj}
    \soumitra{It is important to note that we can derive the Lipschitz continuity of the function $x\rightarrow\nabla^{2}f(x)(x-P_{C_{i}}^{f}(x))$ using Lemma \ref{Breg:lem:BregProj} under the assumptions that $\nabla^2f$ is bounded and Lipschitz continuous, $f$ is twice continuously differentiable and $\sigma$-strongly convex, and there exists a constant $R_i>0$ such that 
    \begin{equation}\label{ineq1:1:1:Breg:lem:BregProj}
        \|x-P_{C_i}^f(x)\|\leq R_i,\;\text{for all}\;x\in\mathbb{R}^n.
    \end{equation}
    However, for a fixed nonempty, closed, and convex subset $C_i$ of $\mathbb{R}^n$, the latter assumption is too restrictive. Indeed, it only holds for $C_i=\mathbb{R}^n$. To see this, suppose $C_i$ is a proper subset of $\mathbb{R}^n$.  Consider the Bregman function $f(x)=(1/2)\|x\|^2$. Then $P_{C_i}^f(x)=P_{C_i}(x)$, the orthogonal projection of $x$ onto $C_i$. Let $x\notin C_i$ and $p_i=P_{C_i}(x)$. Assume that $v_i=x-p_i\neq 0.$ It is observed that $p_i=P_{C_i}(p_i+tv_i)$, $t\geq0$, which implies that 
    \begin{eqnarray}
        \text{dist}(p_i+tv_i, C_i)=\|p_i+tv_i-p_i\|=t\|v_i\|\rightarrow\infty\;\text{as}\; t\rightarrow\infty.
    \end{eqnarray}
    This contradicts the assumption in \eqref{ineq1:1:1:Breg:lem:BregProj}.}
\end{remark}

\soumitra{The following example illustrates the assumptions in \eqref{ineq1:Lemma:Breg:Diff3} and \eqref{ineq2:Lemma:Breg:Diff3} of Lemma \ref{Lemma:Breg:Diff3}.}

\begin{example}\label{Example:Lemma:Breg:Diff3}
    \soumitra{Let $f(x)=(1/2)\|x\|^2$, $x\in\mathbb{R}^n$. It is observed that $\nabla^2f(x)=I$ and $P_{C_i}^f(x)=P_{C_i}(x)$. Therefore,
    $\nabla^2f(x)(x-P_{C_i}^f(x))=(I-P_{C_i})(x)$. It is well-known that $I-P_{C_i}$ is nonexpansive for every nonempty, closed, and convex set $C_i$, i.e.,
    \begin{eqnarray}\label{ineq1:Example:Lemma:Breg:Diff3}
        \|(I-P_{C_i})(x)-(I-P_{C_i})(y)\|\leq \|x-y\|,\;\text{for all}\; x,y\in\mathbb{R}^n.
    \end{eqnarray}
    Hence, the assumption in \eqref{ineq1:Lemma:Breg:Diff3} holds with $L_1^i=1$.}

    \soumitra{Similarly, assume that $g_j(y)=(1/2)\|y\|^2$, $y\in\mathbb{R}^{m_j}$. Then $\nabla^2 g_j(y)=I$ and $P_{Q_h^j}^{g_j}(y)=P_{Q_h^j}(y)$. Since $I-P_{Q_h^j}$ is nonexpansive for every nonempty, closed, and convex set $Q_h^j$,
    \begin{eqnarray}\label{ineq2:Example:Lemma:Breg:Diff3}
        &\|&\nabla^{2}g_j(A_{j}x)(A_{j}x-P_{Q_{h}^{j}}^{g_j}(A_{j}x))-\nabla^{2}g_j(A_{j}y)(A_{j}y-P_{Q_{h}^{j}}^{g_j}(A_{j}y))\|\nonumber\\
        &\leq& \|A_j(x-y)\|\nonumber\\
        &\leq& \|A_j\|_2\|x-y\|.
    \end{eqnarray}
    Hence, the assumption in \eqref{ineq2:Lemma:Breg:Diff3} holds with $L_2^{j,h}=\|A_j\|_2$}.
\end{example}

\begin{remark}\label{rmk1:Example:Lemma:Breg:Diff3}
 \soumitra{In addition, let $C_i=\{x\in\mathbb{R}^n\vert\|x\|\leq 1\}$ in Example \ref{Example:Lemma:Breg:Diff3}, that is, let $C_i$ be the closed ball centered at $0$ with radius $1$, then 
 \begin{eqnarray}\label{ineq1:rmk1:Example:Lemma:Breg:Diff3}
 P_{C_i}(x)=
     \begin{cases}
         x,&\|x\|\leq 1\\
         \frac{x}{\|x\|},&\|x\|> 1.
     \end{cases}
 \end{eqnarray}
 Therefore, 
 \begin{eqnarray}\label{ineq2:rmk1:Example:Lemma:Breg:Diff3}
 x-P_{C_i}(x)=
     \begin{cases}
         0,&\|x\|\leq 1\\
         x-\frac{x}{\|x\|},&\|x\|> 1.
     \end{cases}
 \end{eqnarray}
 Then the assumption in \eqref{ineq1:Lemma:Breg:Diff3} still holds with $L_1^i=1$. However, for $\|x\|>1$
 \begin{eqnarray}\label{ineq5:Example:Lemma:Breg:Diff3}
     \|x-P_{C_i}(x)\|=\Big(1-\frac{1}{\|x\|}\Big)\|x\|,
 \end{eqnarray}
which is unbounded.}

\soumitra{Similarly, let
 \[
A_j=
\begin{pmatrix}
I_n\\[1mm]
0_{(m_j-n)\times n}
\end{pmatrix}
\in\mathbb{R}^{m_j\times n},
\qquad m_j\geq n,
\]
$Q_h^j(y)=\{0\}\subset\mathbb{R}^{m_j}$, in Example \ref{Example:Lemma:Breg:Diff3}.
Then the assumption in \eqref{ineq2:Lemma:Breg:Diff3} holds with $L_2^{j,h}=\|A_j\|_2=1$. Moreover, $P_{Q_h^j}(y)=0$ and $A_jx=(x,0)$. Thus, 
\begin{eqnarray}\label{ineq3:rmk1:Example:Lemma:Breg:Diff3}
    \|A_jx-P_{Q_h^j}(A_jx)\|=\|A_jx\|=\|x\|,
\end{eqnarray}
which is unbounded.}
\end{remark}

\begin{remark}\label{rmk:Breg:beforeAlgo2}  
We observed that the problem \eqref{SFPMOS} is equivalent to the
unconstrained minimization problem: 
\begin{eqnarray}
\begin{aligned}\min_{x\in\mathbb{R}^{n}}p_{f,g}(x).\end{aligned}
\label{min:gen:breg}
\end{eqnarray}
\end{remark}

\begin{breakablealgorithm}\caption{\soumitra{A Generalized Simultaneous Projection Gradient Algorithm (GSPGA) for SFPMOS}}
\label{Algo:GSPGA} \begin{algorithmic}[1]\vspace{-0.2cm}
\State \textbf{Input}: \soumitra{$x^{0}\in\mathbb{R}^{n}$.} 
\State \soumitra{Compute $L_{\nabla p_{f,g}}:=\sum_{i=1}^{r}\alpha_{i}L_1^i+\sum_{j=1}^{t}\sum_{h=1}^{\ell_{j}}\beta_{h}^{j}L_2^{j,h}\|A_{j}\|_{2}$.}
\State Choose $\lambda\in(0,2/L_{\nabla p_{f,g}})$. 
\soumitra{\For{$k=0,1,2,\dots$}
\State Compute 
\begin{eqnarray}\label{grad:breg:gen:algo}
\begin{aligned}
x^{k+1}:=x^k-\lambda\sum_{i=1}^{r}\alpha_{i}\nabla^{2}f(x^k)&(x^k-P_{C_{i}}^{f}(x^k))\\&-\lambda\sum_{j=1}^{t}\soumitra{\sum_{h=1}^{\ell_{j}}}\beta_{h}^{j}A_{j}^{T}\nabla^{2}g_j(A_{j}x^k)(A_{j}x^k-P_{Q_{h}^{j}}^{g_j}(A_{j}x^k)).
\end{aligned}
\end{eqnarray}
\EndFor}
\end{algorithmic} 
\end{breakablealgorithm}

\begin{theorem}\label{MainThm:2} \soumitra{Assume that all the hypotheses of Lemma \ref{Lemma:Breg:Diff3} hold. Let the solution set $\widetilde{\varTheta}$
of the problem \eqref{SFPMOS} is nonempty. Then any
sequence $\{x^{k}\}_{k=0}^{\infty}$ generated by Algorithm \ref{Algo:GSPGA}
converges in norm to a point in $\widetilde{\varTheta}$.}
\end{theorem}

\begin{proof} \soumitra{Since $\widetilde{\varTheta}$ is nonempty, the problem \eqref{SFPMOS} is consistent. 
By Corollary \ref{prox:SFPMO} and Remark \ref{rmk:Breg:beforeAlgo2}, the solution set of \eqref{SFPMOS} coincides with the global minimizers of the unconstrained minimization problem in \eqref{min:gen:breg} and the minimum value of $p_{f,g}$ is zero. Moreover, by Lemma \ref{Lemma:Breg:Diff3}, the function $p_{f,g}$ is continuously
differentiable, and its gradient is $L_{\nabla p_{f,g}}$-Lipschitz continuous. 
By the given hypotheses, the function $p_{f,g}$ is convex.
Since the step-size $\lambda\in(0,2/L_{\nabla p_{f,g}})$, all the assumptions of Corollary \ref{Co1:Grad:thm} are satisfied. By the expression in \eqref{grad:breg:gen:algo},
our Algorithm \ref{Algo:GSPGA} is a standard gradient descent algorithm
for the problem \eqref{SFPMOS}. Therefore,
by applying Corollary \ref{Co1:Grad:thm}, we conclude that the sequence $\{x^{k}\}_{k=0}^{\infty}$
generated by Algorithm \ref{Algo:GSPGA} converges in norm to a point $x^*\in\arg\min_{x\in\mathbb{R}^n} p_{f,g}(x)$.
Since the set of all global minimizers of $p_{f,g}$ coincides with $\widetilde{\varTheta}$, $x^*\in\widetilde{\varTheta}$. Thus, the sequence $\{x^k\}_{k=0}^\infty$ converges in norm to a solution of the problem \eqref{SFPMOS}.}
\end{proof}

\begin{remark}\label{rmk:bregman:Algo:GSPA} 
 Algorithm \ref{Algo:GSPGA} is new and distinct from the algorithm
studied in \cite[Subsection 4.1]{SREI2022} for the problem \eqref{SFPMOS}.
\end{remark}


\section{Other Important Special Cases of \eqref{MSMSICFP}}

\label{other:special:MSMSICFP}

Specializing from point-to-set operators to single-valued operators
brings us back to many split problems, some of which we list here. 
\begin{enumerate}
\item \textbf{The split common fixed point problem}: \soumitra{Let $\mathbb{R}^{m_{j}}=\mathbb{R}^{m}$
for all $j=1,2,\dots,t$ and each point-to-set operator $C_{i}=S_{i}:\mathbb{R}^{n}\rightarrow\mathbb{R}^{n}$
be a single-valued nonexpansive operator for each $i=1,2,\dots,r$.
Let each point-to-set operator $Q^{j}_{h}=\Xi_{j}:\mathbb{R}^{m}\rightarrow\mathbb{R}^{m}$
be a single-valued nonexpansive operator for $h=1,2,\dots,\ell_{j}$,
$j=1,2,\dots,t$. Let $A_{j}=A:\mathbb{R}^{n}\rightarrow\mathbb{R}^{m}$
be a bounded linear operator for each $j=1,2,\dots,t$.} Then the
problem \eqref{MSMSICFP} consists of finding a point $x^{*}\in\mathbb{R}^{n}$
such that 
\begin{align}
x^{*}\in\bigcap\limits^{r}_{i=1}\text{Fix}(S_{i})\;\text{ and }\;A(x^{*})\in\bigcap\limits^{t}_{j=1}\text{Fix}(\Xi_{j}).\tag{SCFPP}\label{SCFPP}
\end{align}
The problem \eqref{SCFPP} was studied in \cite[Subsection 4.2]{SREI2022}
for infinite-dimensional Hilbert spaces.
\item \textbf{The fixed point problem}: Let $\mathbb{R}^{m_{j}}=\mathbb{R}^{n}$
for all $j=1,2,\dots,t$ and each point-to-set operator $C_{i}=S_{i}:\mathbb{R}^{n}\rightarrow\mathbb{R}^{n}$
be a single-valued nonexpansive operator for each $i=1,2,\dots,r$.
Assume that $t=r$. Let each point-to-set operator $Q^{j}_{h}=C_{j}$
be a single-valued nonexpansive operator and $A_{j}=I$, the identity
operator, for all $h=1,2,\dots,\ell_{j}$, $j=1,2,\dots,t$.
Then the problem \eqref{MSMSICFP} consists of finding a point $x^{*}\in\mathbb{R}^{n}$
such that 
\begin{align}
x^{*}\in\bigcap\limits^{r}_{i=1}\text{Fix}(S_{i}).\tag{FPP}\label{FPP}
\end{align}
The problem \eqref{FPP} was studied in \cite[Subsection 4.3]{SREI2022}
for infinite-dimensional Hilbert spaces.
\item \textbf{The multiple-sets split feasibility problem}: Let $\mathbb{R}^{m_{j}}=\mathbb{R}^{m}$,
$Q^{j}_{h}=Q_{j}$ be a nonempty, closed, and convex subset
of $\mathbb{R}^{m}$, $A_{j}=A$ for $h=1,2,\dots,\ell_{j}$, $j=1,2,\dots,t$,
and for $i=1,2,\dots,r$, $C_{i}=\Omega_{i}$ be a nonempty, closed,
and convex subset of $\mathbb{R}^{n}$. 
Then the problem \eqref{MSMSICFP} reduces to \textit{multiple-sets
split feasibility problem} (MSSFP), which was introduced by Censor
et al. \cite[Subsection 1.1]{YCEN2005}. That is, the problem \eqref{MSMSICFP}
then consists of finding a point $x^{*}\in\mathbb{R}^{n}$ such that
\begin{align}
x^{*}\in\Omega:=\bigcap\limits^{r}_{i=1}\Omega_{i}\;\text{ and }\;Ax^{*}\in Q:=\bigcap\limits^{t}_{j=1}Q_{j}.\tag{MSSFP}\label{MSSFP}
\end{align}
\item \textbf{The split feasibility problem}: Let $\mathbb{R}^{m_{j}}=\mathbb{R}^{m}$,
$Q^{j}_{h}=Q$ be a nonempty, closed, and convex subset of $\mathbb{R}^{m}$,
$A_{j}=A$ for $h=1,2,\dots,\ell_{j}$, $j=1,2,\dots,t$,
and for $i=1,2,\dots,r$, $C_{i}=C$ be a nonempty, closed, and convex
subset of $\mathbb{R}^{n}$. Then the problem \eqref{MSMSICFP} reduces
to the classical {\em split feasibility problem} (SFP), which was
first introduced and studied by Censor et al. \cite[Section 1 and Subsection 6.1]{YCEN1994}.
That is, the problem \eqref{MSMSICFP} then consists of finding a
point $x^{*}\in\mathbb{R}^{n}$ such that 
\begin{align}
x^{*}\in C\;\text{ and }\;Ax^{*}\in Q.\tag{SFP}\label{SFP}
\end{align}
\item \textbf{The convex feasibility problem}: Let $t=r$, $\mathbb{R}^{m_{j}}=\mathbb{R}^{n}$,
$Q^{j}_{h}=\Omega_{j}$ be nonempty, closed, and convex subsets of
$\mathbb{R}^{n}$, $A_{j}=I$, the identity operator on $\mathbb{R}^{n}$,
for $h=1,2,\dots,\ell_{j}$, $j=1,2,\dots,t$, and for $i=1,2,\dots,r$,
$C_{i}=\Omega_{i}$ be nonempty, closed, and convex subsets of $\mathbb{R}^{n}$.
Then the problem \eqref{MSMSICFP} reduces to {\em convex feasibility
problem} (CFP). That is, the problem \eqref{MSMSICFP} then consists
of finding a point $x^{*}\in\mathbb{R}^{n}$ such that 
\begin{align}
x^{*}\in\Omega:=\bigcap\limits^{r}_{i=1}\Omega_{i}.\tag{CFP}\label{CFP}
\end{align}
The problem \eqref{CFP} was studied by many researchers (see, for
example, \cite{HHBA1996} and \cite[Chapter 5]{YCEN1997}).
\item \soumitra{\textbf{The split common minimization problem:} Let
$f_{i}:\mathbb{R}^{n}\rightarrow\mathbb{R}\cup\{+\infty\}$, $i=1,2,\dots,r$,
and $g_{j}:\mathbb{R}^{m}\rightarrow\mathbb{R}\cup\{+\infty\}$, $j=1,2,\dots,t$
be proper, lower semi-continuous, and convex functions. Let $\mathbb{R}^{m_{j}}=\mathbb{R}^{m}$
and $A_{j}=A:\mathbb{R}^{n}\rightarrow\mathbb{R}^{m}$ be a bounded
linear operator for all $j=1,2,\dots,t$. If we set $C_{i}(x)=\arg\min_{u\in\mathbb{R}^{n}}f_{i}(u)\neq\varnothing$,
$i=1,2,\dots,r$, $Q^{j}_{h}(y)=\arg\min_{v\in\mathbb{R}^{m}}g_{j}(v)\neq\varnothing$,
$h=1,2,\dots,\ell_{j}$, $j=1,2,\dots,t$, then the problem \eqref{MSMSICFP}
reduces to the following minimization problem: find $x^{*}\in\mathbb{R}^{n}$
such that 
\begin{eqnarray}
\begin{aligned}x^{*}\in\bigcap^{r}_{i=1}\arg\min_{u\in\mathbb{R}^{n}}f_{i}(u)\;\;\text{and}\;\;Ax^{*}\in\bigcap^{t}_{j=1}\arg\min_{v\in\mathbb{R}^{m}}g_{j}(v).\end{aligned}
\end{eqnarray}
This problem is called the \textit{split minimum point problem} in
\cite[Subsection 4.1]{SREI2020}, where it is studied in infinite-dimensional
Hilbert spaces.} 
\end{enumerate}

\section{Conclusion}

\label{Sec:conclusion}

\soumitra{We introduced the multi-split multiple-sets implicit convex
feasibility problem (MSMS-ICFP), which provides a general framework
for feasibility-seeking problems involving implicit constraint sets
in the domain space and multiple split implicit constraint sets in
possibly different image spaces. The formulation extends the implicit
convex feasibility problem and many other well-known inverse split
problems. We employed a proximity function to reformulate the problem
\eqref{MSMSICFP} as an unconstrained minimization problem. We analyzed
various results on the properties of the proximity function and established
convergence of the sequence generated by the proposed simultaneous
projection gradient method in the consistent case. Additionally, we
presented a Bregman-distance-based generalization. Since the domain
and image constraints belong to different Euclidean spaces, the generalization
uses appropriate Bregman functions in the relevant spaces. For the
split feasibility problem with multiple output sets, this led to a
generalized simultaneous projection gradient method. We established
norm convergence of the sequence generated by the algorithm in the
consistent case. The special cases of the problem \eqref{MSMSICFP}
are presented in this work to illustrate the scope of the proposed framework
and its relation to various existing models in feasibility, fixed
point theory, and split inverse problems. Moreover, they show that
the problem \eqref{MSMSICFP} can serve as a common setting for studying
apparently different problems. During our analysis, a number of open
research questions emerge. And, of course, future numerical experiments
and applications would be useful for demonstrating the practical performance
of the proposed methods.}




\medskip{}

\textbf{Acknowledgments}. This work is supported by U.S. National Institutes of
Health (NIH) Grant Number R01CA266467 and by the Cooperation Program
in Cancer Research of the German Cancer Research Center (DKFZ) and
Israel’s Ministry of Innovation, Science and Technology (MOST).




\end{document}